\documentclass[12pt,reqno]{amsart}
\usepackage{graphicx}
\usepackage{amssymb,amsmath}
\usepackage{amsthm}
\usepackage{color,graphicx}
\usepackage{hyperref}
\usepackage{color}
\usepackage{epstopdf}
\usepackage{lpic}
\usepackage{enumerate}
\usepackage{mathrsfs}

\newcommand{\be}{\begin{equation}}
\newcommand{\ee}{\end{equation}}

\newcommand{\R}{{\mathbb R}}

\newcommand{\Bs}{\mathscr{B}_{\sigma}}
\renewcommand{\P}{\mathscr P}

\numberwithin{equation}{section}
\numberwithin{figure}{section}

\newtheorem{theorem}{Theorem}[section]

\newtheorem{remark}[theorem]{Remark}
\newtheorem{lemma}[theorem]{Lemma}
\newtheorem{corollary}[theorem]{Corollary}

\begin{document}

\title[Smooth periodic waves in the Degasperis--Procesi equation]{Stability of smooth periodic traveling waves \\ in the Degasperis--Procesi equation}

\author{Anna Geyer}
\address[A. Geyer]{Delft Institute of Applied Mathematics, Faculty Electrical Engineering, Mathematics and Computer Science, Delft University of Technology, Mekelweg 4, 2628 CD Delft, The Netherlands}
\email{A.Geyer@tudelft.nl}

\author{Dmitry E. Pelinovsky}
\address[D.E. Pelinovsky]{Department of Mathematics and Statistics, McMaster University,	Hamilton, Ontario, Canada, L8S 4K1}
\email{dmpeli@math.mcmaster.ca}

\date{\today}
\maketitle

\begin{abstract} 
We derive  a precise energy stability criterion for smooth periodic waves in the Degasperis--Procesi (DP) equation. Compared to the Camassa-Holm (CH) equation, the number of negative eigenvalues of an associated Hessian operator changes in the existence region of smooth perodic waves. We utilize properties of the period function with respect to 
two parameters in order to obtain a smooth existence curve for the family
of smooth periodic waves of a fixed period. The energy stability condition  
is derived on parts of this existence curve which correspond to either one or two negative eigenvalues of the Hessian operator. We show numerically 
that the energy stability condition is satisfied on either part of the curve and prove analytically that it holds in a neighborhood of the boundary of the existence region of smooth periodic waves. 
\end{abstract} 

\section{Introduction}

The Degasperis-Procesi (DP) equation 
\begin{equation}
\label{DP}
u_t-u_{txx}+4uu_x = 3u_xu_{xx}+uu_{xxx}
\end{equation}
has a special role in the modeling of fluid motion. It was derived in \cite{dp} as a transformation of the integrable hierarchy of KdV equations, with the same asymptotic accuracy as the Camassa--Holm (CH) equation \cite{Cam}. Although a more general family of model equations can also be derived by using this method  \cite{dhh,Dullin}, only the DP and CH equations are integrable with the use of the inverse scattering transform. It was shown in 
\cite{Const4,rossen,Johnson} that the DP and CH equations describe the horizontal velocity $u = u(t,x)$ for the unidirectional propagation of waves of a shallow water flowing over a flat bed at a certain depth. A review of applicability of these model equations as approximations of  peaked waves in fluids was recently given in \cite{Lund}.

In the present paper, we are concerned with smooth traveling wave solutions, for which the DP and CH equations have been justified as model equations in hydrodynamics \cite{Const4}. Existence of smooth periodic 
traveling waves has been well understood by using ODE methods \cite{Len1,Len2}. However, stability of smooth periodic traveling waves 
was considered to be a difficult problem in the functional-analytic framework, 
even though integrability implies their stability  
due to the structural stability of the Floquet spectrum of the associated linear system \cite{Len3}. Only very recently in \cite{GMNP}, we derived an energy stability criterion for the smooth periodic traveling waves of the CH equation by using its Hamiltonian formulation. 

For smooth solitary waves, orbital stability  was obtained for the CH equation in \cite{CS-02} 
and spectral and orbital stability for the DP equation was obtained in \cite{Liu-21, Liu-22}. The energy stability criterion 
for the smooth solitary waves was derived for the entire family of the generalized CH equations \cite{LP-22} and was shown to be satisfied asymptotically and numerically. A recent work \cite{Long} 
used the period function to show that the energy stability criterion is satisfied analytically for the entire family of smooth solitary waves.\\

{\em The purpose of this work is to derive an energy stability criterion 
for the smooth periodic traveling waves in the DP equation.} \\

Let us briefly comment on the various Hamiltonian formulations which exist both for the CH and DP equations. These two equations belong to a larger class of  generalized CH equations, the so-called $b$-family, which reduces to CH for $b=2$ and to DP for $b=3$. As far as we know, only one Hamiltonian formulation exists for general $b$, which was obtained in \cite{dhh-proc}  and used in the stability analysis of smooth solitary waves in \cite{LP-22}, while one more (alternative) formulation exists for $b=3$ and two more alternative formulations exist for $b=2$. In \cite{GMNP}, we used the two alternative formulations to study spectral stability of the smooth periodic waves. Here we will only use the alternative formulation which exists for $b = 3$. Whether the Hamiltonian formulation from \cite{dhh-proc} can also be adopted to the study of spectral stability of smooth periodic waves for the $b$-family is left for further studies. 

We  consider the DP equation (\ref{DP}) in the periodic domain $\mathbb{T}_L := [0,L]$ of length $L > 0$. For notational simplicity, we write 
$H^s_{\rm per}$ instead of $H^s(\mathbb{T}_L)$ for the Sobolev space of $L$-periodic functions with index $s \geq 0$. The DP equation (\ref{DP}) on $\mathbb{T}_L$ formally conserves the mass, momentum, and energy given respectively by 
\begin{equation}
\label{Mu}
M(u)= \oint u dx,
\end{equation}
\begin{equation}
\label{Eu}
E(u)=\frac{1}{2} \oint  u (1-\partial_x^2) (4 - \partial_x^2)^{-1} u dx,
\end{equation}
and
\begin{equation}
\label{Fu}
F(u)=\frac{1}{6} \oint u^3 dx.
\end{equation}
The standard Hamiltonian structure for the DP equation (\ref{DP}) is given by 
\begin{equation}
\label{sympl-1}
\frac{du}{dt} = J \frac{\delta F}{\delta u}, \quad 
J = -(1-\partial_x^2)^{-1} (4-\partial_x^2) \partial_x, 
\end{equation}
where $J$ is a well-defined operator from $H^{s+1}_{\rm per}$ to $H^{s}_{\rm per}$ for every $s \geq 0$ and $\frac{\delta F}{\delta u} = \frac{1}{2} u^2$. The evolution problem (\ref{sympl-1}) is well-defined for local solutions
$u \in C((-t_0,t_0),H^s_{\rm per}) \cap C^1((-t_0,t_0),H^{s-1}_{\rm per})$ with $s > \frac{3}{2}$, see \cite{Escher2008}, where $t_0 > 0$ is the local existence time.

Smooth traveling waves of the form 
$u(t,x) = \phi(x-ct)$ with $c - \phi > 0$ 
are obtained from the critical points of the augmented 
energy functional
\begin{equation}
\label{action-1}
\Lambda_{c,b}(u) := c E(u) - F(u) - \frac{b}{4} M(u),
\end{equation}
where $b$ is a parameter obtained after integration of the third-order differential equation \eqref{third-order} satisfied by the traveling wave profile $\phi$, see Section \ref{sec-2}. After two integrations of the third-order equation (\ref{third-order}) with integration constants $a$ and $b$, all smooth periodic wave solutions 
with the profile $\phi$ can be found from the first-order 
invariant
\begin{equation}
\label{quadra}
(c-\phi)^2 (\phi'^2 - \phi^2 - b) + a = 0.
\end{equation}

The second variation of the augmented energy functional (\ref{action-1}) is determined by an associated Hessian operator  $\mathcal{L} : L^2_{\rm per}  \to L^2_{\rm per}$ given by 
\begin{equation}
\mathcal{L} := c - \phi - 3 c (4 - \partial_x^2)^{-1}.
\label{hill} 
\end{equation} 
The  operator $\mathcal L$ is  self-adjoint and bounded as the sum of the bounded multiplication operator $(c-\phi)$ and the compact operator $- 3 c (4 - \partial_x^2)^{-1}$ in $L^2_{\rm per}$. Since $c - \phi > 0$, the continuous spectrum of $\mathcal{L}$ is strictly positive, hence $\mathcal{L}$ has finitely many negative eigenvalues of finite algebraic multiplicities and a zero eigenvalue of finite algebraic multiplicity.

The first result of this paper is about the existence of smooth periodic traveling waves with profile $\phi$ satisfying the first-order invariant (\ref{quadra}), and the number  of negative eigenvalues of $\mathcal{L}$ given by (\ref{hill}).

\begin{theorem}
	\label{theorem-existence}
	For a fixed $c > 0$, smooth periodic solutions of the first-order invariant (\ref{quadra}) with  profile $\phi \in H^{\infty}_{\rm per}$ satisfying $c - \phi > 0$ exist in an open, simply connected region on the $(a,b)$ plane enclosed by three boundaries: 
	\begin{itemize}
		\item $a = 0$ and $b \in (-c^2,0)$, where the periodic solutions are peaked, 
		\item $a = a_+(b)$ and $b \in (0,\frac{1}{8} c^2)$, where the solutions have infinite period,
		\item $a = a_-(b)$ and $b \in (-c^2,\frac{1}{8} c^2)$, where the solutions are constant, 
	\end{itemize}
	where $a_+(b)$ and $a_-(b)$ are  smooth functions of $b$. For every point inside the region, the periodic solutions are smooth functions of $(a,b)$ and their period is strictly increasing in $b$ for every fixed $a \in (0,\frac{27}{256} c^4)$. There exists a smooth curve $a = a_0(b)$ for $b \in (-\frac
{2}{9}c^2,0)$ in the interior of the existence region such that the Hessian operator $\mathcal{L}$ in $L^2_{\rm per}$ has only one simple negative eigenvalue above the curve and two simple negative eigenvalues (or a double negative eigenvalue) below the curve. The rest of its spectrum  for $a \neq a_0(b)$ includes a simple zero eigenvalue and a strictly positive spectrum bounded away from zero. Along the curve $a = a_0(b)$ the Hessian operator $\mathcal{L}$ has only one simple negative eigenvalue, a double zero eigenvalue, and the rest of its spectrum is strictly positive. 
\end{theorem}

\begin{figure}[htb!]
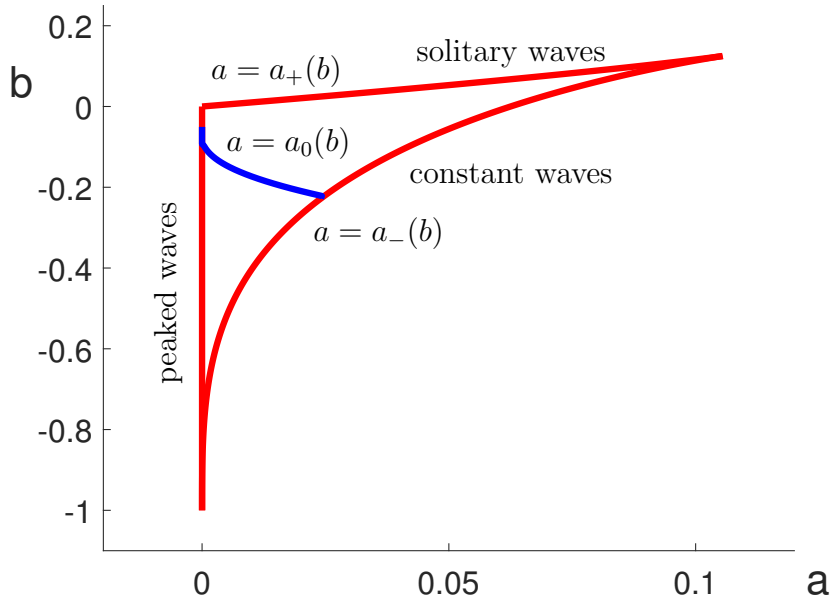

	\begin{lpic}{region(0.8,0.8)}
		\lbl{87,95;solitary waves}
		\lbl{48,92; $a=a_+(b)$}
		\lbl{50,80; $a=a_0(b)$}
		\lbl{87,75; constant waves}
		\lbl{65,65; $a=a_-(b)$}
		\lbl{30,55;\rotatebox{90}{peaked waves}}
	\end{lpic}
	\caption{The existence region of smooth periodic solutions 
		of the first-order invariant (\ref{quadra}) 
		on the parameter plane $(a,b)$ for $c = 1$ enclosed by three boundaries (red lines). The blue line shows the curve $a = a_0(b)$ which separates the cases of one and two negative eigenvalue of $\mathcal{L}$.} \label{fig-domain}	
\end{figure} 

The three  curves bounding  the existence region of smooth periodic waves in Theorem \ref{theorem-existence} are shown in Figure \ref{fig-domain} for $c = 1$. The curve in the interior of the existence region is the curve $a = a_0(b)$, which was found numerically by plotting the period function of the periodic solutions of Theorem \ref{theorem-existence} versus $a$ for fixed $b$ and detecting its maximum if it exists, see Lemmas \ref{lem-nonmonotonicity} and \ref{lem-Floquet} below. 

The  transformation 
\begin{equation}
\label{scaling}
\phi(x) = c \varphi(x), \quad b = c^2 \beta, \quad a = c^4 \alpha
\end{equation}
normalizes the parameter $c$ to unity with $\varphi$, $\beta$, and $\alpha$ satisfying the same equation (\ref{quadra})  but with $c = 1$. Hence, the smooth periodic waves are uniquely determined by the free parameters $(a,b)$ and $c = 1$ can be used everywhere. Similarly, 
although we only consider the case of right-propagating waves with $c > 0$, all results can be extended to the left-propagating waves with $c < 0$ 
by using the scaling transformation (\ref{scaling}).

Spectral stability of  smooth periodic travelling waves with respect to co-periodic perturbations is determined 
by the spectrum of the linearized operator $J \mathcal{L}$ in $L^2_{\rm per}$, with $J$ given in \eqref{sympl-1}. 
Since $J$ is a skew-adjoint operator  and $\mathcal{L}$ is self-adjoint, the spectrum of the linearized operator  $J \mathcal{L}$  is symmetric with respect to $i \mathbb{R}$ \cite{HK08}. Therefore,  the periodic wave is spectrally stable if the spectrum of $J \mathcal{L}$ in $L^2_{\rm per}$ is located on $i \mathbb{R}$. The second result of this paper gives the energy criterion for the spectral stability  of the smooth periodic waves in the DP equation (\ref{DP}).

\begin{theorem}
	\label{theorem-stability}
	For a fixed $c > 0$ and a fixed period $L > 0$, there exists 
	a $C^1$ mapping $a \mapsto b = \mathcal{B}_L(a)$ for $a \in (0,a_L)$ with  some $L$-dependent $a_L \in (0,\frac{27}{256} c^4)$ 
	and a $C^1$ mapping $a \mapsto \phi = \Phi_L(\cdot,a) \in H^{\infty}_{\rm per}$ of smooth $L$-periodic solutions along the curve $b = \mathcal{B}_L(a)$.	Let 
	$$
	\mathcal{M}_L(a) := M(\Phi_L(\cdot,a)) \quad \mbox{\rm and} \quad 
	\mathcal{F}_L(a) := F(\Phi_L(\cdot,a)),
	$$
	where $M(u)$ and $F(u)$ are given by (\ref{Mu}) and (\ref{Fu}). 
	The $L$-periodic wave with profile $\phi = \Phi_L(\cdot,a)$ such that 
	$\mathcal{B}_L'(a) \neq 0$ is spectrally stable if the mapping
	\begin{equation}
	\label{stability-criterion}
	a \mapsto \frac{\mathcal{F}_L(a)}{\mathcal{M}_L(a)^3}
	\end{equation}
	is strictly decreasing and,  for $\mathcal{B}_L'(a) < 0$, if additionally the mapping $a \mapsto \mathcal{M}_L(a)$ is strictly increasing. The stability criterion holds true for every point in a neighborhood of the boundary $a = a_-(b)$.
\end{theorem}

\begin{figure}[htb!]
	\includegraphics[width=10cm,height=8cm]{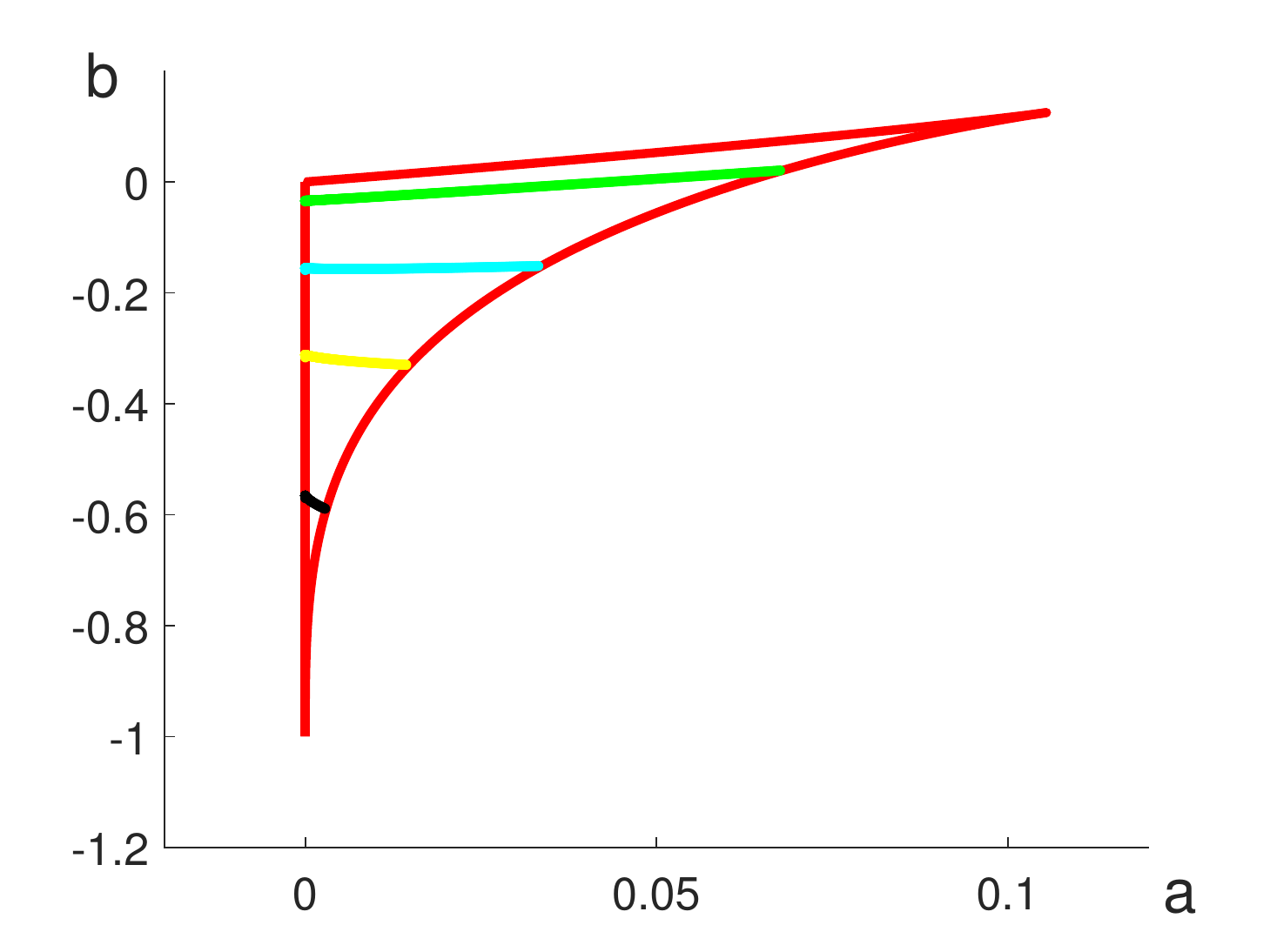}
	\includegraphics[width=7cm,height=6cm]{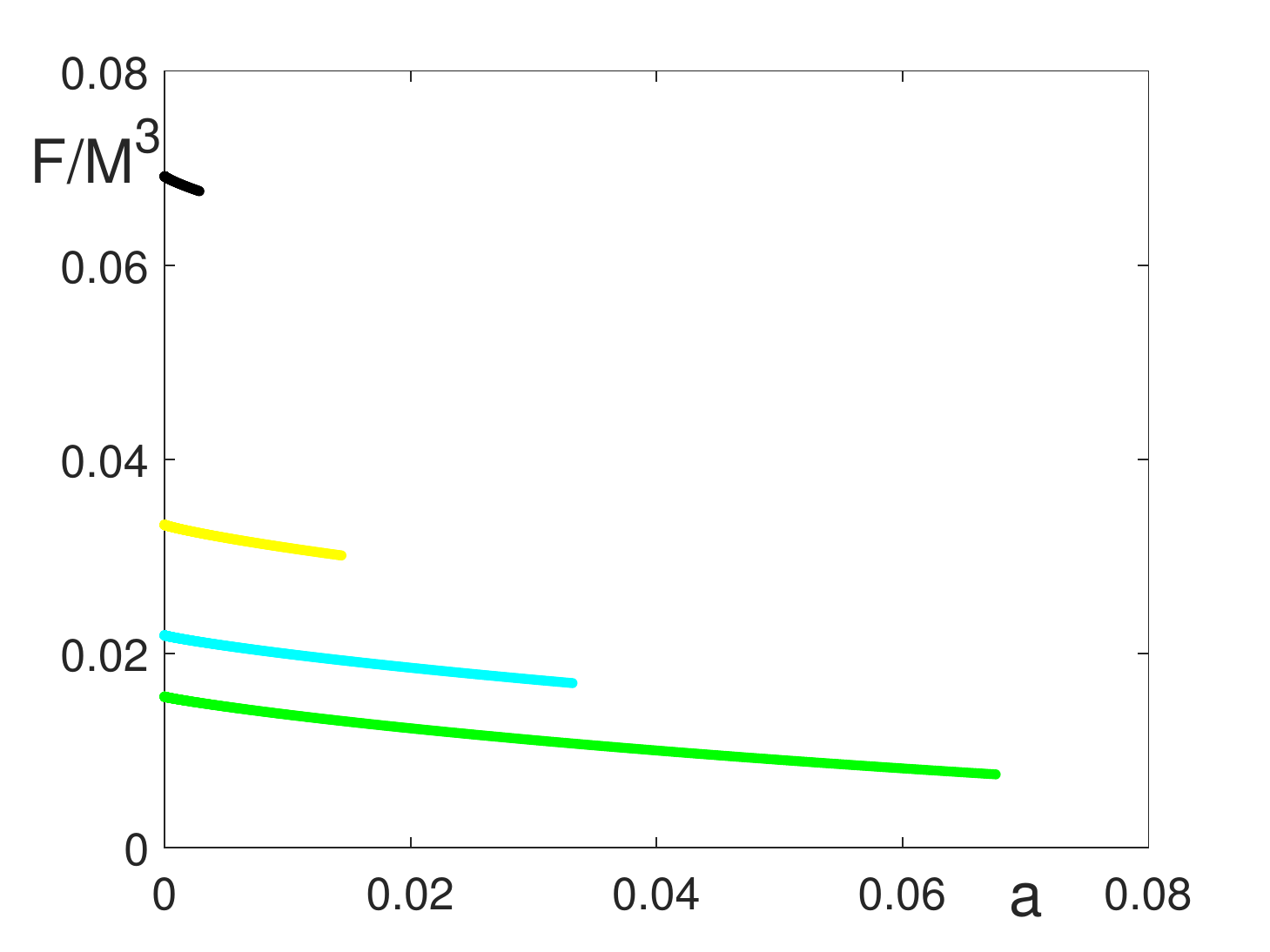}
	\includegraphics[width=7cm,height=6cm]{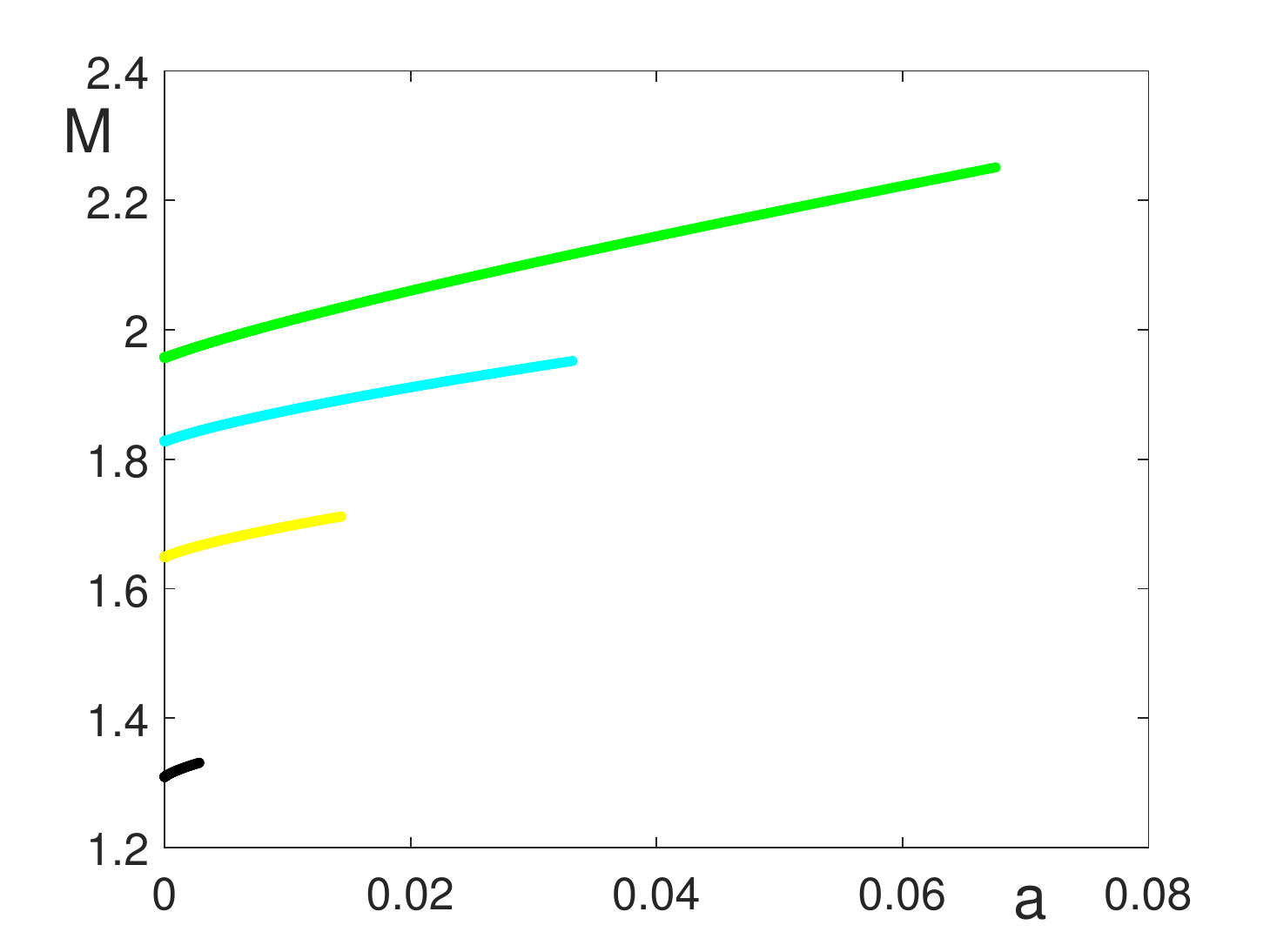}
	\caption{Top: Existence region on the $(a,b)$ plane with four curves 
	$b = \mathcal{B}_L(a)$ for $c = 1$ and four values of period: $L = \pi/2$ (black), $L = 3\pi/4$ (yellow), $L = \pi$ (cyan), and $L = 3 \pi/2$ (green).
Bottom: $\mathcal{F}_L/\mathcal{M}_L^3$ versus $a$ (left) and  $\mathcal{M}_L$ versus $a$ (right) along the four curves.} 
	\label{fig-dependence}
\end{figure}

Figure \ref{fig-dependence} shows the numerically computed mappings $a \mapsto \mathcal{F}_L(a)/\mathcal{M}_L^3(a)$ and $a \mapsto \mathcal{M}_L(a)$
for four values of fixed $L$. The parameter $a$ is chosen in $(0,a_L)$, 
where $a_L$ depends on $L$. It follows that the stability criterion of Theorem \ref{theorem-stability}  is satisfied for all cases. This property has been analytically proven only near the boundary $a = a_-(b)$ by means of  
the Stokes expansion, see Lemma \ref{lem-Stokes}.

It is harder to check the stability criterion of Theorem \ref{theorem-stability} near the other two boundaries of the existence region of Theorem \ref{theorem-existence} where the waves are either peaked or solitary. 
The perturbation theory becomes singular in these two 
asymptotic limits because $c - \phi$ vanishes for the peaked periodic waves 
and the period function diverges for the  solitary waves. 
Nevertheless, some relevant results are available in these two  limits:  

\begin{itemize}
	\item For the boundary $a = 0$ and $b \in (-c^2,0)$, where the periodic solutions are peaked, the spectral stability problem for the DP equation (\ref{DP}) needs to be set up by using a weak formulation of the evolution problem. This setup was elaborated for a generalized CH equation in \cite{LP-21}, building on previous work in \cite{Natali}, to show spectral instability of peaked solitary waves. Linear and nonlinear instability of peaked periodic waves with respect to peaked periodic perturbations was shown for the CH equation in \cite{MP-2021}. Spectral and linear instability of peaked periodic waves for the reduced Ostrovsky equation was proven in \cite{Geyer-Ostrovsky,Geyer-Pel2019}. Instability of peaked periodic waves in the DP equation or in the generalized CH equation is still open for further studies. \\
	
	\item For the boundary $a = a_+(b)$ and $b \in (0,\frac{1}{8}c^2)$, where  the periodic solutions have infinite period, spectral stability of solitary waves over a nonzero background was shown for the general $b$-family in \cite{LP-22} and for the DP equation in \cite{Liu-21}. The methods 
	in \cite{LP-22,Liu-21} are not related to the energy stability criterion (\ref{stability-criterion}), and it remains open to show the equivalence of the three different stability criteria for  smooth solitary waves over a nonzero background. 
\end{itemize}

The analytical proof of the energy stability criterion of Theorem \ref{theorem-stability} in the interior of the bounded existence region is still open. 
Another interesting open question is to explore the non-standard 
Hamiltonian formulation of the DP equation as a member of the 
$b$-family  and to obtain a different energy stability criterion 
for the smooth periodic waves. Finally, there may exist 
a deep connection between the energy stability criterion 
and the physical laws for fluids since the mapping (\ref{stability-criterion}) 
involves a homogeneous function of degree zero in terms of the wave 
profile $\phi$. Similarly, the energy stability criterion 
for the CH equation obtained in \cite{GMNP} involves a homogeneous 
function of degree zero given by $\mathcal{E}_L(a)/\mathcal{M}_L(a)^2$, 
where $\mathcal{M}_L(a)$ is the same as in \eqref{Mu} and $\mathcal{E}_L(a)$ is obtained from  $E(u)= \| u \|^2_{H^1_{\rm per}}$, which is different from (\ref{Eu}).\\

The paper is organized as follows. In section \ref{sec-2} we state and prove the existence result for the smooth periodic wave with profile $\phi$,  similar to \cite{GMNP} and \cite{LP-22}. Section \ref{sec-4} details the monotonicity properties of the period function for the smooth periodic solutions of DP with respect to both parameters $a$ and $b$. The proofs rely on the classical works \cite{Chic,Vil2014} but involve more complicated details of computations compared to \cite{GMNP, Geyer2015b} for the CH equation. Section \ref{sec-3} describes the number of negative eigenvalues and the multiplicity of the zero eigenvalue of the Hessian operator $\mathcal{L}$. The count is obtained by a nontrivial adaptation of the Birman--Schwinger principle which is different from the study of a similar Hessian operator for solitary waves in \cite{Liu-21}. The proof of Theorem \ref{theorem-existence} is achieved with the results obtained in Sections \ref{sec-2}, \ref{sec-4}, and \ref{sec-3}. Finally, in Section \ref{sec-5} we extend the family of periodic waves with the profile $\phi$ along a curve with a fixed period $L > 0$ and give the proof of Theorem \ref{theorem-stability}. \\

{\bf Acknowledgement.} This project was completed in June 2022 during a Research in Teams stay at the Erwin Schr\"{o}dinger Institute, Vienna. The authors thank Yue Liu for many discussions related to this project. D. E. Pelinovsky acknowledges the funding of this study provided by Grants No. FSWE-2020-0007 and No. NSH-70.2022.1.5.

\section{Smooth traveling waves} 
\label{sec-2}

Traveling waves of the form $u(t,x) = \phi(x-ct)$ with speed $c$ and profile $\phi$ are found from the third-order differential equation
\begin{equation}
\label{third-order}
-(c-\phi) (\phi''' - \phi') - 3 \phi \phi' + 3 \phi' \phi'' = 0,
\end{equation}
which is obtained from the DP equation (\ref{DP}). For notational convenience 
we denote $\phi = \phi(x)$ where $x$ stands for the traveling coordinate 
$x - ct$. Integration of (\ref{third-order}) in $x$ gives the second-order equation
\begin{equation}
\label{CHode}
-(c-\phi)\phi''+c\phi + \phi'^2 -2\phi^2 = b,
\end{equation}
where $b$ is an integration constant. Another second-order equation can be obtained after multiplying (\ref{third-order}) 
by $(c-\phi)^2$ and integrating,
\begin{equation}
\label{second-order}
-(c - \phi)^3 (\phi'' - \phi) = a,
\end{equation}
where $a$ is another integration constant. Both second-order equations (\ref{CHode}) and (\ref{second-order}) are compatible if and only if 
$\phi$ satisfies the first-order invariant (\ref{quadra}), 
which can  be viewed as the first-order invariant for 
either (\ref{CHode}) or (\ref{second-order}). 

The following lemma characterizes the family of periodic waves 
by using  phase plane analysis, and constitutes the existence 
part of Theorem \ref{theorem-existence}. 

\begin{lemma}
	\label{lem-trav}
	For a fixed $c > 0$, smooth periodic solutions to the first-order invariant (\ref{quadra}) with the profile $\phi \in H^{\infty}_{\rm per}$ satisfying $c - \phi > 0$ exist in an open, simply connected region on the $(a,b)$ plane enclosed by three boundaries: 
\begin{itemize}
	\item $a = 0$ and $b \in (-c^2,0)$, where the periodic solutions are peaked, 
	\item $a = a_+(b)$ and $b \in (0,\frac{1}{8} c^2)$, where the solutions have infinite period,
	\item $a = a_-(b)$ and $b \in (-c^2,\frac{1}{8} c^2)$, where the solutions are constant, 
\end{itemize}
where $a_+(b)$ and $a_-(b)$ are smooth functions of $b$. The family of periodic solutions inside this region is smooth in $(a,b)$. 
\end{lemma}

\begin{proof}
	For a fixed $c > 0$, the first-order invariant (\ref{quadra}) 
	represents the energy conservation $(\phi')^2 + U(\phi) = b$ 
	for a Newtonian particle 
	with  mass $m = 2$ and the energy level $b$ under a force with 
	the potential energy
	\begin{equation*}
	U(\phi) := -\phi^2 + \frac{a}{(c-\phi)^{2}}.
	\end{equation*}
	The critical points of $U$ on $\mathbb{R}\backslash \{c\}$ are given by the roots of the algebraic equation 
	$$
	\phi (c-\phi)^3 = a.
	$$ 
	The global maximum of $\phi \mapsto \phi (c - \phi)^3$ occurs at $\phi = \phi_c := c/4$ for which $a = a_c := \frac{27}{256} c^4$. 
	
\begin{figure}[htb!]
	\includegraphics[width=0.48\textwidth]{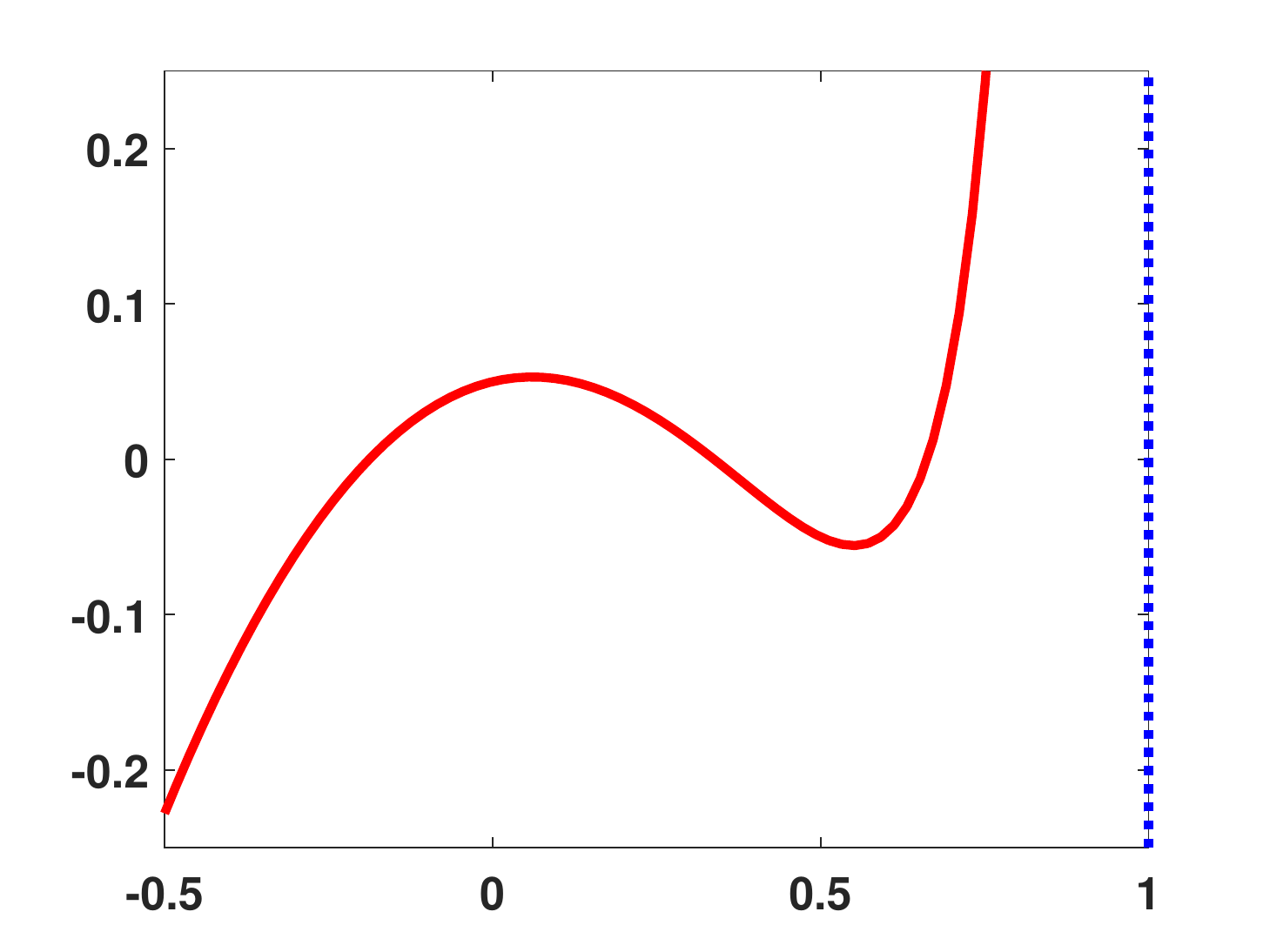}
	\includegraphics[width=0.48\textwidth]{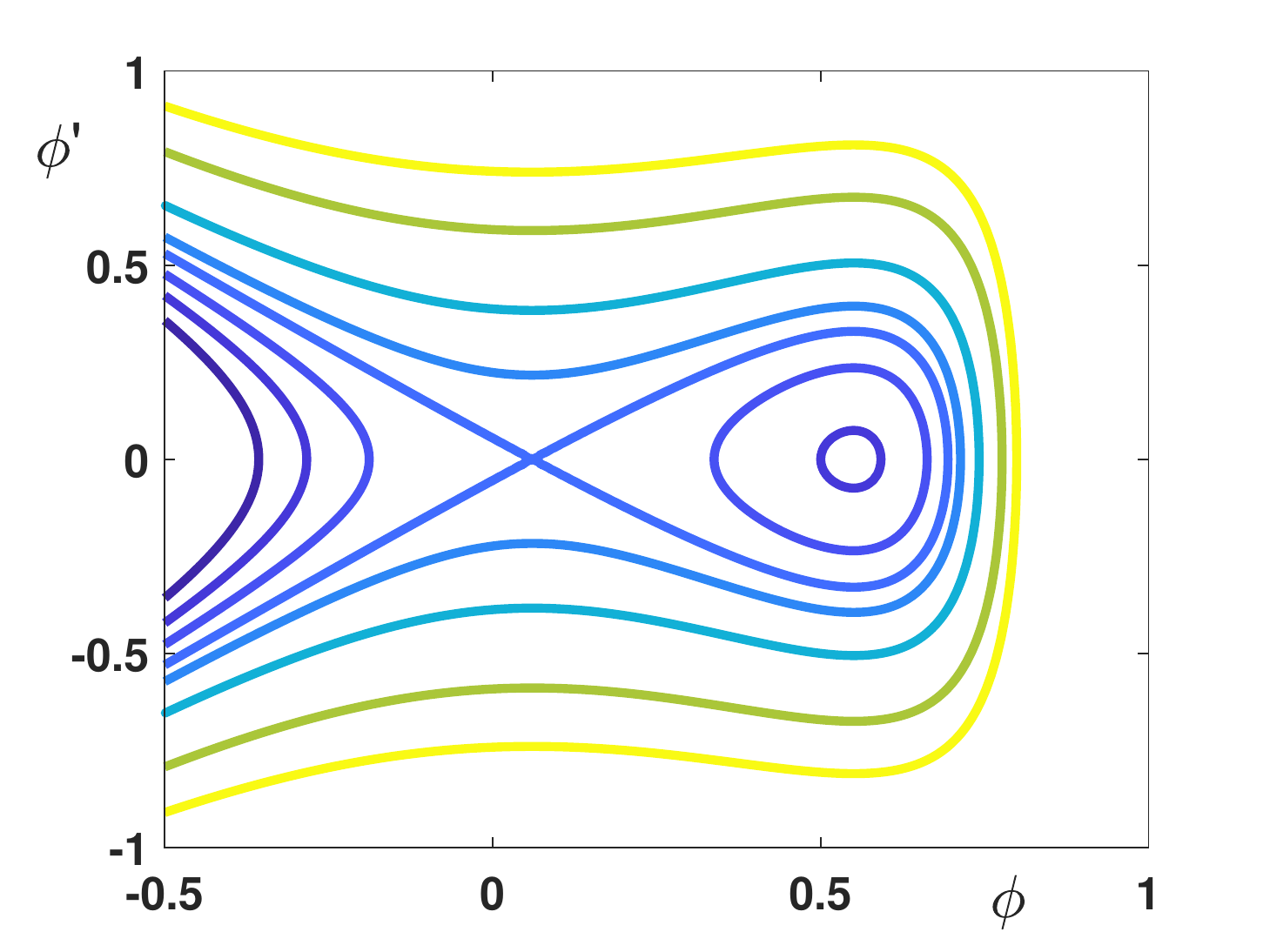}
	\caption{Left: $U$ versus $\phi$ for $c = 1$, and $a = 0.05$. Right: the phase portrait of the second-order equation (\ref{second-order}) constructed from the level curves of $b = (\phi')^2 + U(\phi)$ on the phase plane $(\phi,\phi')$ for the same parameter values.} \label{fig-plane}
\end{figure}
	
If $a \in (0,a_c)$, 
the potential energy $U$ has a local maximum $\phi_1$ and a local minimum $\phi_2$ which satisfy the ordering 
\begin{equation}
\label{ordering}
0 < \phi_1 < \frac{c}{4} < \phi_2 < c,
\end{equation}	
see the left panel of Figure \ref{fig-plane}. The local maximum and minimum  of $U$ give  the saddle point $(\phi_1,0)$ and the center point $(\phi_2,0)$ of the first-order planar system corresponding to the second-order equation (\ref{second-order}). 
Smooth periodic solutions with the profile $\phi$ satisfying $c-\phi > 0$ correspond to  periodic orbits inside a  punctured neighbourhood around the center $(\phi_2,0)$ enclosed by the homoclinic orbit connecting the saddle $(\phi_1,0)$, see the right panel of Figure \ref{fig-plane}. All other orbits are unbounded. 

If $a \in (-\infty,0)$, the potential energy $U$ has two local maxima, 
one is below the singularity at $c$ and the other one is above the singularity at $c$ with $U(\phi) \to -\infty$ as $\phi \to c$. All orbits are either unbounded or hit the singularity at $c$ for which $\phi'$ is infinite. 
The same is true for $a = 0$, for which $U(\phi) = -\phi^2$.

If $a \in [a_c,\infty)$, the potential energy does not have local extremal points and $U(\phi) \to +\infty$ as $\phi \to c$. All   orbits of the second-order equation (\ref{second-order}) are unbounded. 

Thus, bounded periodic solutions  exist if and only if $a \in (0,a_c)$. 
Note that $\phi$ depends smoothly on the parameters $a$ and $b$ in view of smooth dependence of the first-order invariant (\ref{quadra}) on $\phi$, $a$, and $b$  if $c - \phi > 0$.

It remains to characterize the three boundaries of the existence region, see Figure \ref{fig-domain}.
If $a = 0$, the second-order equation (\ref{second-order}) becomes $\phi'' - \phi = 0$ and is solved  explicitly by the $L$-periodic solution
\begin{equation*}
\phi(x) = c \frac{\cosh\left(\frac{L}{2}-|x|\right)}{\cosh\left(\frac{L}{2}\right)}, \quad x \in \left[-\frac{L}{2},\frac{L}{2}\right],
\end{equation*}
which attains the singularity $\phi = c$ placed at $x = 0$. As a result, 
the $L$-periodic wave is peaked at $x = 0$ and smooth at $x = \pm \frac{L}{2}$ with $\phi'\left(\pm \frac{L}{2}\right) = 0$. 
It follows from $b = (\phi')^2 - \phi^2$ that 
\begin{equation}
\label{dependence-b-L}
b = -c^2 {\rm sech}^2\left(\frac{L}{2}\right), 
\end{equation}
so that $b \in (-c^2,0)$ for $L \in (0,\infty)$. 

If $a \in (0,a_c)$, 
the periodic orbit exists for the energy level $b \in (b_-,b_+)$, where 
$b_- = U(\phi_2)$ and $b_+ = U(\phi_1)$. On each respective boundary, $a$ and $b$ can be parameterized by $\phi_2 \in (\phi_c,c)$ and  $\phi_1 \in (0,\phi_c)$, where $\phi_c = c/4$. The periodic solution along 
$b = b_-(a)$ is constant and we have 
\begin{equation}
\label{a-minus}
\left\{ \begin{array}{l} 
b = c \phi_2 - 2 \phi_2^2,\\
a = \phi_2 (c - \phi_2)^3,
\end{array} \right. \qquad \Rightarrow \qquad 
\left\{ \begin{array}{l} 
\frac{db}{d\phi_2} = c - 4 \phi_2 < 0,\\
\frac{da}{d\phi_2} = (c - \phi_2)^2 (c - 4 \phi_2) < 0.
\end{array} \right.
\end{equation}
Hence, in view of the chain rule, $b = b_-(a)$ is a monotonically increasing function, which can be inverted to obtain a function $a = a_-(b)$ for $b \in (-c^2,\frac{1}{8}c^2)$. Similarly, along $b = b_+(a)$, the periodic solution degenerates into a homoclinic solution of infinite period and we have 
\begin{equation*}
\left\{ \begin{array}{l} 
b = c \phi_1 - 2 \phi_1^2,\\
a = \phi_1 (c - \phi_1)^3,
\end{array} \right. \qquad \Rightarrow \qquad 
\left\{ \begin{array}{l} 
\frac{db}{d\phi_1} = c - 4 \phi_1 > 0,\\
\frac{da}{d\phi_1} = (c - \phi_1)^2 (c - 4 \phi_1) > 0.
\end{array} \right.
\end{equation*}
Hence $b = b_+(a)$ is a monotonically increasing function, which can be inverted to obtain a function $a = a_+(b)$ for $b \in (0,\frac{1}{8}c^2)$. 
\end{proof}

Next we show that the  periodic traveling wave with  profile $\phi$ 
is a critical point of the augmented energy functional $\Lambda_{c,b}$ defined in (\ref{action-1}).

\begin{lemma}
	\label{lem-var}
Let $\phi \in H^{\infty}_{\rm per}$ be an $L$-periodic solution of the first-order invariant (\ref{quadra}) for some $(a,b)$ inside the existence region specified in Lemma \ref{lem-trav} for fixed $c > 0$. Then, $\phi$ is a critical point of the augmented energy functional $\Lambda_{c,b}$.
\end{lemma}

\begin{proof}
	Let us introduce the momentum variable $m := u - u_{xx}$ and 
	the auxillary variable $v := (4-\partial_x^2)^{-1} u$ associated with the velocity variable $u$. For the traveling wave $u(t,x) = \phi(x-ct)$, we write $m(t,x) = \mu(x-ct)$ and $v(t,x) = \nu(x-ct)$ so that 
	$$
	\mu := \phi - \phi'', \quad \nu := (4 - \partial_x^2)^{-1} \phi.
	$$
	It follows from the second-order equation (\ref{second-order}) that $\mu = a (c-\phi)^{-3}$, whereas the third-order equation (\ref{third-order}) is equivalent to 
	$$
	c (\phi' - \phi''') - (4-\partial_x^2) \phi \phi' = 0,
	$$
which by virtue of $4 \nu - \nu'' = \phi$  gives the  relation 
	\begin{equation*}
	c(\nu' - \nu''') - \phi \phi' = 0.
	\end{equation*}
Integration yields
	\begin{equation}
	\label{nu-phi-equation}
	c(\nu - \nu'') - \frac{1}{2} \phi^2 = d,
	\end{equation}
	where $d$ is an integration constant. Since
	$4 \nu - \nu'' = \phi$, we obtain from (\ref{nu-phi-equation}) that 
	\begin{equation}
	\label{v-phi}
	\nu = \frac{1}{3} \phi - \frac{1}{6c} \phi^2 - \frac{d}{3c}.
	\end{equation}
	Substituting (\ref{v-phi}) into $\phi = 4 \nu - \nu''$ 
	and expressing $\phi''$ and $(\phi')^2$ by using (\ref{second-order}) and 
	(\ref{quadra}) gives us the relation 
	$d = b/4$ between the integration constants.
The Euler--Lagrange equation for $\Lambda_{c,b}$ is given by 
\begin{equation}
\label{EL}
c (1-\partial_x^2) (4 - \partial_x^2)^{-1} \phi - \frac{1}{2} \phi^2 - \frac{b}{4} = 0,
\end{equation}
which therefore coincides with (\ref{nu-phi-equation}). By Lemma \ref{lem-trav}, the periodic solutions of the first-order invariant (\ref{quadra}) are smooth, so that they are also smooth solutions of the Euler--Lagrange equation (\ref{EL}) and hence the critical points of 
$\Lambda_{c,b}$.
\end{proof}

\begin{remark}
The statement of Lemma \ref{lem-var} does not work in the opposite direction, 
since critical points of $\Lambda_{c,b}$ are solutions of the Euler--Lagrange equation (\ref{EL}) which are only defined in the weak space $L^{\infty}(\mathbb{T}_L)$. In particular, the set of critical points of $\Lambda_{c,b}$ includes the peaked periodic waves which occur at the boundary $a = 0$ of the existence region for smooth periodic waves in Lemma \ref{lem-trav}.
\end{remark}

\begin{remark}
In the variables $m := u - u_{xx}$ and $v := (4-\partial_x^2)^{-1} u$, 
the DP equation (\ref{DP}) can be rewritten in the local forms
\begin{equation*}
m_t + u m_x + 3m u_x = 0
\end{equation*}
and
\begin{equation*}
v_t - v_{txx} + u u_x = 0.
\end{equation*}
Traveling wave reductions of these equations give relations between $\phi$, 
$\mu$, and $\nu$, which appear in the proof of Lemma \ref{lem-var}.
\end{remark}

\begin{remark}
	The Hessian operator $\mathcal{L}= \Lambda_{c,b}''(\phi)$ 
	given by (\ref{hill}) is not related to the linearization of the second-order equations (\ref{CHode}) and (\ref{second-order}). It is related to the linearization of the second-order equation (\ref{nu-phi-equation}) in the sense that the derivative of (\ref{nu-phi-equation}) in $x$ yields 
	$$
c (1-\partial_x^2) \nu' - \phi \phi' = 0
	$$
	which implies $\mathcal{L} \phi' = 0$ in view of the fact that $c(1 - \partial_x^2) (4 - \partial_x^2)^{-1} = c - 3c (4 - \partial_x^2)^{-1}$.
\end{remark}

\section{Period function}
\label{sec-4}

Here we shall study monotonicity properties of the period function for the smooth periodic solutions of Lemma \ref{lem-trav}  with respect to parameters $a$ and $b$ for fixed $c > 0$. For $a \in (0,a_c)$, where $a_c := \frac{27}{256} c^4$, we let $\phi_+$ and $\phi_-$ be the turning points for which $U(\phi_{\pm}) = b$ for each $b \in (b_-,b_+)$. It follows from the proof of Lemma \ref{lem-trav}, see Figure \ref{fig-plane}, that the turning points fit into the ordering (\ref{ordering}) as follows:
\begin{equation*}
0 < \phi_1 < \phi_- < \frac{c}{4} < \phi_2 < \phi_+ < c.
\end{equation*}
The period function $\mathfrak{L}(a,b)$ assigns to each smooth periodic solution of the first-order invariant (\ref{quadra}) its  fundamental period $L=\mathfrak{L}(a,b)$. Rewriting (\ref{quadra}) in the form  
$$
(\phi')^2 + U(\phi) = b
$$ 
and integrating it along the periodic orbit $\phi$, it follows  that the period function is given by 
\begin{equation}
\label{period-L-function}
\mathfrak{L}(a,b) := 2 \int_{\phi_-}^{\phi_+} \frac{d\phi}{\sqrt{b - U(\phi)}}
\end{equation}
for every point $(a,b)$ inside the existence region of Lemma \ref{lem-trav}.

\subsection{Monotonicity of the period function with respect to the parameter $b$}

We shall prove that the period function $\mathfrak{L}(a,b)$ is a strictly increasing function of $b$ for fixed $c > 0$ and $a \in (0,a_c)$. This gives the second result of Theorem \ref{theorem-existence}.

\begin{lemma}
	\label{theorem-increasing}
	Fix $c > 0$ and $a \in (0,a_c)$
	The period function $\mathfrak{L}(a,b)$ is strictly increasing as a function of $b$.
\end{lemma}

\begin{proof}
Recall that $\phi_2$ is the local minimum of $U$ and hence the second root of the algebraic equation $a = \phi_2 (c-\phi_2)^3$, see the ordering (\ref{ordering}), which we may use to replace the parameter $a$. Then, using the transformation $\{x=\frac{\phi-\phi_2}{\phi_2}, \;\; y=\frac{\phi'}{\phi_2}\}$, we can write the second-order equation (\ref{second-order}) as the planar system 
	\begin{equation}\label{sys}
	\left\{
	\begin{array}{l}
	x' =y,\\[2pt]
	y' =1+x-\frac{\eta^3}{(\eta - x)^3},
	\end{array}\right.
	\end{equation}
	associated with the Hamiltonian 
	\begin{equation}
	\label{potential-H}
	H(x,y)=\frac{y^2}{2} + V(x), \quad V(x) := -\frac{x^2}{2} - x -\frac{\eta}{2}+ \frac{\eta^3}{2(\eta-x)^2},
	\end{equation}
	where $\eta = \frac{c-\phi_2}{\phi_2}\in(0,3)$.	The potential $V$ is smooth away from the singular line $x=\eta$, has a local minimum at $x=0$ and a local maximum at $x_1 < 0$, see Figure \ref{fig-potential-V}.
	
		\begin{figure}[htb!]
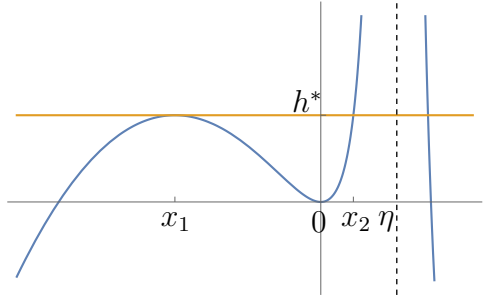

		\begin{lpic}{potential1(0.5,0.5)}
			\lbl{45,20;$x_1$}
			\lbl{83,20;$0$}
			\lbl{93,20;$x_2$}		
			\lbl{101,20;$\eta$}
			\lbl{80,52;$h^*$}
		\end{lpic}
		\caption{The potential function $V(x)$ plotted for $\eta =0.5$.} 
		\label{fig-potential-V}
	\end{figure}

	The center at the origin is surrounded by periodic orbits $\gamma_h$, which lie inside the level curves $H(x,y) = h$ with $h\in(0,h^*)$ and $h^*=V(x_1)$. Denote by $x_2$ the unique solution of $V(x_1)=V(x)$ such that $x_1<0<x_2<\eta$, see Figure \ref{fig-potential-V}. Finally, define the period function of the center $(0,0)$ of system (\ref{sys}) by 
	\begin{equation*}
	\ell(h)=\int_{\gamma_h} \frac{dx}{y} \quad \text{ for each } h\in(0,h^*).
	\end{equation*}
	Note that $b = 2 \phi_2^2 h + \phi_2 (c - 2 \phi_2)$ 
	and $\mathfrak{L}(a,b) = \ell(h)$ for fixed $a \in (0,a_c)$ and $c > 0$.
	Since $\phi_2$ is fixed, we have $\partial_b \mathfrak{L}(a,b) > 0$ if and only if $\ell'(h) > 0$.

	To prove that $\ell'(h) > 0$, we shall use a monotonicity criterion by Chicone \cite{Chic} for planar systems with Hamiltonians of the form (\ref{potential-H}), where $V$ is
	a smooth function on $(x_1,x_2)$ with a nondegenerate relative minimum at the origin. According to the main theorem in \cite{Chic} the period function $\ell(h)$ is monotonically increasing in $h$ if the function 
	$$
	W(x) := \frac{V(x)}{(V'(x))^2}
	$$
	is convex in $(x_1,x_2)$. Hence, we have to prove that $W''(x) > 0$ for every $x \in (x_1,x_2)$. A straightforward computation shows that 
	\begin{equation*}
	W''(x)=-\frac{3 (\eta - x)^2 R_\eta(x)}{(\eta^2 (3 - \eta) + 3 \eta (\eta - 1) x + (1-3\eta) x^2 + x^3)^4},
	\end{equation*}
	where
	\begin{eqnarray*}
		R_\eta(x) &=& (\eta-1) x^6 + 10 \eta (1-\eta) x^5 + 5 \eta^2 (7 \eta - 9) x^4 + 20 \eta^3 (5 - 3 \eta) x^3 \\
		&& + 5 \eta^3 (11 \eta^2 - 22 \eta -1) x^2 
		+ 2 \eta^4 (9 + 28 \eta - 13 \eta^2) x + 5 \eta^5 (\eta^2 - 2 \eta - 3).
	\end{eqnarray*}
	We need to show that $R_\eta(x) < 0$ for $x \in (x_1,x_2)$ and  $\eta \in (0,3)$.
	The case  $\eta=1$ has to be considered separately, for which we find that 
	$$
	R_1(x) = -10 (x-1)^4 - 2(5-4x) < 0 \text{ for  } x \in (-\infty,1), 
	$$
	and, in particular, $R_1(x)<0$ on $(x_1,x_2)$ since $x_1<0<x_2<\eta$. For $\eta \in I :=(0,3)\setminus\{1\}$ we will use a result  for univariate polynomials depending on a parameter, see \cite[Lemma 8.1]{GGG}, which can be used to ensure that the polynomial $R$ does not change sign on the interval $\Omega:=(x_1,x_2)$ when varying the parameter $\eta$.	
	In what follows, $\emph{\mbox{Res}}_x(f,g)$ stands for the \emph{multipolynomial resultant} of two polynomials $f$ and $g$ in $x$ (see for instance~\cite{Cox2007,Fulton1984}).
	
	We will check the assumptions of \cite[Lemma 8.1]{GGG} one by one. 
	Assumption (i) clearly holds for all $\eta \in I$. To prove assumption (ii) we compute the discriminant of $R_\eta$ with respect to $x$,
	\begin{equation*}
	\mbox{Disc}_x(R_\eta)= -6400000\eta^{23}(\eta + 1)^4(\eta - 1)(27\eta^2 + 14\eta + 3)(\eta^3 - 5\eta^2 + 11\eta + 1)
	\end{equation*}
	and see that it is different from zero  on $I$ since the term $\eta^3 - 5\eta^2 + 11\eta + 1$ has only a negative real root and all other terms do not have any real roots on $I$. To ensure assumption (iv) we need to show that $R_{\eta}$ does not vanish in the boundary points $x_1,x_2$ of $\Omega$. Since we do not have explicit expressions for these points  we will compare $R_{\eta}$ with other polynomials with explicitly known roots since $x_1$ is the nontrivial zero of $V'(x)$ and $x_2$ is the unique zero of $V(x_1)-V(x)$ in $(x_1,\eta)$. Since 
	\begin{equation*}
	\mbox{Res}_x(R_\eta,V')=5\eta^{15}(\eta + 1)^2(\eta - 3)^2(27\eta^2 + 14\eta + 3)^2 \neq 0,
	\end{equation*}
	for  $\eta \in I$, the polynomials $R_\eta$ and $V'$ do not have a common root, so in particular $R_\eta(x_1)\neq 0 $ for $\eta \in I$. For the other boundary point $x_2$, we define 
	\begin{equation*}
	P_\eta(x):=\mbox{Res}_y(V'(y), V(y)-V(x))	
	\end{equation*}
	and compute 
	\begin{equation*}
	\mbox{Res}_x(R_\eta,P_\eta) = 25\,{\eta}^{45} \left( \eta+1 \right) ^{7} \left( \eta-3 \right) ^{6}
	\left( 9\,{\eta}^{2}+21\,\eta+16 \right)  \left( 27\,{\eta}^{2}+14\,
	\eta+3 \right) ^{2} Q(\eta).
	\end{equation*}
where $Q$ is a polynomial of degree $34$ whose expression we omit. We use Sturm's method, see \cite[Theorem 5.6.2]{SB}, to prove that it has one root at $\eta=\eta_1\approx 1.083$, and find the rational lower and upper bounds 
	$$
	\eta_1 \in [\underbar{$\eta$}_1,\bar\eta_1]:=\left[ \frac{277}{256}, \frac{555}{512}\right]\subset I.
	$$ 
	This proves that $R_\eta(x_2)\neq 0$ for $\eta \in I\setminus\{\bar\eta\}$. Therefore, the number of zeros of $R_\eta(x)$ on $\Omega$ is constant for $\eta \in I\setminus\{\eta_1\}$. The value $\eta = \eta_1$ is treated separately at the end of the proof. Finally, to ensure assumption (iii) we have to show that there exists some $\eta$ in each of the  subintervals of $I\setminus\{\eta_1\}$ such that $R_\eta(x)\neq 0$. For $\eta=\frac{1}{2} \in (0,1)$, 
	\begin{equation*}
	R_{\frac{1}{2}}(x)= \frac{1}{128}\left(-64x^6 + 320x^5 - 880x^4 + 1120x^3 - 740x^2 + 316x - 75\right).
	\end{equation*}
	Using again Sturm's method we can show that $R_{\frac{1}{2}}$ has  two real roots $r_i$, $i=1,2$, for which we can find rational lower and upper bounds such that $r_i \in [\underbar{r}_i,\bar r_i]=:I_i$, $i=1,2$, for instance 
	$$
	r_1\in I_1= \left[ \frac{327}{512}, \frac{655}{1024}\right] \quad 
	\mbox{\rm and} \quad r_1 < r_2 \in I_2 = \left[ \frac{991}{1024}, \frac{31}{32}\right].
	$$ 
	To show that the two roots are outside of $\Omega$ we use Sturm's method once more for the polynomial $P_{\frac{1}{2}}(x)$ to find rational bounds for 
	$$
	x_2 \in [\underbar{$x$}_2, \bar x_2]:= \left[ \frac{94993}{131072}, \frac{23749}{32768}\right].
	$$ 
	Then it is straightforward to see that $V(\bar x_2) -V(\underbar{$r$}_1) <0$, which implies that $x_2<\bar x_2 < \underbar{$r$}_1<r_1<r_2$ since $V$ is monotone increasing for $x>0$. Hence $R_{\frac{1}{2}}(x)\neq 0$. Similarly, we show that $R_\eta\neq 0$ for $\eta=\underbar{$\eta$}_1\in (1,\eta_1)$ and $\eta=2\in(\bar\eta_1,3)$. Then, by \cite[Lemma 8.1]{GGG}, $R_\eta(x)\neq 0$ on  $\Omega$ for all $\eta\in I\setminus\{\eta_1\}$ and one can easily check that $R_\eta(x)<0$ on $\Omega$ in each of the subintervals of $ I\setminus\{\eta_1\}$. 
	
	To ensure that  also $R_{\eta_1}(x)<0$ we prove that $R_{\eta}$ is monotone in a neighborhood of $\eta_1$, i.e.~we show that $R'_\eta(x)\neq 0$ on $\Omega$ for $\eta \in (\underbar{$\eta$}_1,\bar \eta_1)$ using again \cite[Lemma 8.1]{GGG}. Indeed, similarly as above we show that $R_\eta(x_1)R_\eta(x_2)\mbox{Disc}_x(R'_\eta)\neq 0$ for $\eta \in (\underbar{$\eta$}_1,\bar \eta_1)$ and evaluating $R'_\eta(x)$ in one value, for instance $\eta =1083/1000\in (\underbar{$\eta_1$}	,\bar \eta_1)$, we find using Sturm's method that $R'_\eta(x)\neq 0$ on $\Omega$ for $\eta \in (\underbar{$\eta_1$},\bar \eta_1)$. 
	
	This concludes the proof that   $W''(x)>0$ for $x \in (x_1,x_2)$ and $\eta \in I$, which yields $\ell'(h) > 0$ by the main theorem in \cite{Chic}. 
\end{proof}

\subsection{Monotonicity of the period function with respect to the parameter $a$}

We shall study monotonicity properties of the period function $\mathfrak{L}(a,b)$ as a function of $a$ for fixed $c > 0$ and $b \in (-c^2,\frac{1}{8}c^2)$. This result will be used to prove the last assertion of  Theorem \ref{theorem-existence}, see Corollary \ref{lem-degenerate}.

\begin{lemma}
	\label{lem-nonmonotonicity}
	Fix $c > 0$ and $b \in (-c^2,\frac{1}{8} c^2)$. The period function 
	$\mathfrak{L}(a,b)$ satisfies the following properties:
	\begin{itemize}
		\item It is strictly monotonically increasing in $a$ if $b \in (-c^2,-\frac{2}{9}c^2]$;
		\item It has a unique critical point in $a$, which is a maximum, if $b \in (-\frac{2}{9}c^2,0)$; 
		\item It is strictly monotonically decreasing in $a$ if $b \in [0,\frac{1}{8} c^2)$.
	\end{itemize}
\end{lemma}

\begin{remark}
The proof of Lemma \ref{lem-nonmonotonicity} follows very closely the one carried out in \cite{Geyer2015b} for the period function of the CH equation and relies strongly on the tools developed in \cite{Vil2014}. For the sake of brevity we refrain from stating the technical details and refer to \cite{Geyer2015b} for more explanations on how these tools are applied.
\end{remark}
 
In contrast to the previous subsection, where periodic smooth traveling waves are characterized as solutions of the second-order equation  \eqref{second-order}, we now regard the traveling waves as solutions of the equivalent second-order equation \eqref{CHode}. For convenience, we rewrite \eqref{CHode} as a planar system such that its center is located at the origin. This is obtained via the change of variables 
$$
\left\{x=\frac{\phi-c}{\sqrt{\Delta}} + \theta, \qquad 
y= \frac{\phi'}{\sqrt{\Delta}} \right\},
$$ 
where $\Delta\!:= c^2-8b>0$ and $\theta\!:= \frac{1}{4}\!\left(\frac{3c}{\sqrt{\Delta}}-1\right)>0$. 
Periodic orbits are obtained from the planar system 
\begin{equation}\label{e-sys_x}
       \left\{
      \begin{array}{l}
     x' =y,\\[2pt]
     y' = \dfrac{x+2x^2 -y^2}{x-\theta},
      \end{array}\right.
\end{equation}
which is analytic away from the singular line $x=\theta$ and has the analytic first integral 
$$	
H(x,y)=A(x) + C(x)y^2,
$$  
with $A(x)=- \frac{1}{6} x^2(3x^2+2x(1-2\theta)-3\theta)$ and 
$C(x) = \frac{1}{2} (x-\theta)^2$. 
The first integral satisfies the hypotheses in \cite[Theorem A]{Vil2014} with $B(x)=0$. Moreover, its integrating factor   $K(x)=(x-\theta)^2$ depends only on $x$. The function $A(x)$ satisfies $A(0)=0$ and has a minimum at $x=0$, which yields a center at $(0,0)$, and two local maxima at $x=\theta$ and $x=-\frac{1}{2}$, the latter one yielding a saddle point at $(-\frac{1}{2},0)$. 
 The period function associated to the center of the differential system \eqref{e-sys_x} can be written as 
 \begin{equation*}
   \ell(h)=\int_{\gamma_h}\frac{dx}{y} \quad \text{ for } \; h\in(0,h^*), 	 	
 \end{equation*}
 where $\gamma_h$ is the  periodic  orbit inside the energy level 
 $\{ (x,y) : \;\; H(x,y) = h\}$ with either $h^*= A(-\frac{1}{2})$ for $\theta\in[\frac{1}{2},\infty)$ or $h^*= A(\theta)$ for $\theta\in (0,\frac{1}{2})$.

When $\theta\geq\frac{1}{2}$, we find that $A(\theta)\geq A(-\frac{1}{2})$, in which  case the period annulus $\mathscr P$ is bounded by the homoclinic orbit at the saddle point, see the right panel of Figure \ref{fig-theta}. When $\theta\in(0,\frac{1}{2})$ the outer boundary of $\mathscr P$ consists of a trajectory with $\alpha$ and $\omega$-limit in the straight line $\{x=\theta\}$ and the segment between these two points, see the left panel of Figure \ref{fig-theta}. In view of these structural differences, we will study the monotonicity of the period function separately for $\theta\in(0,\frac{1}{2})$ and  for $\theta\geq \frac{1}{2}$.

\begin{figure}[htb!]
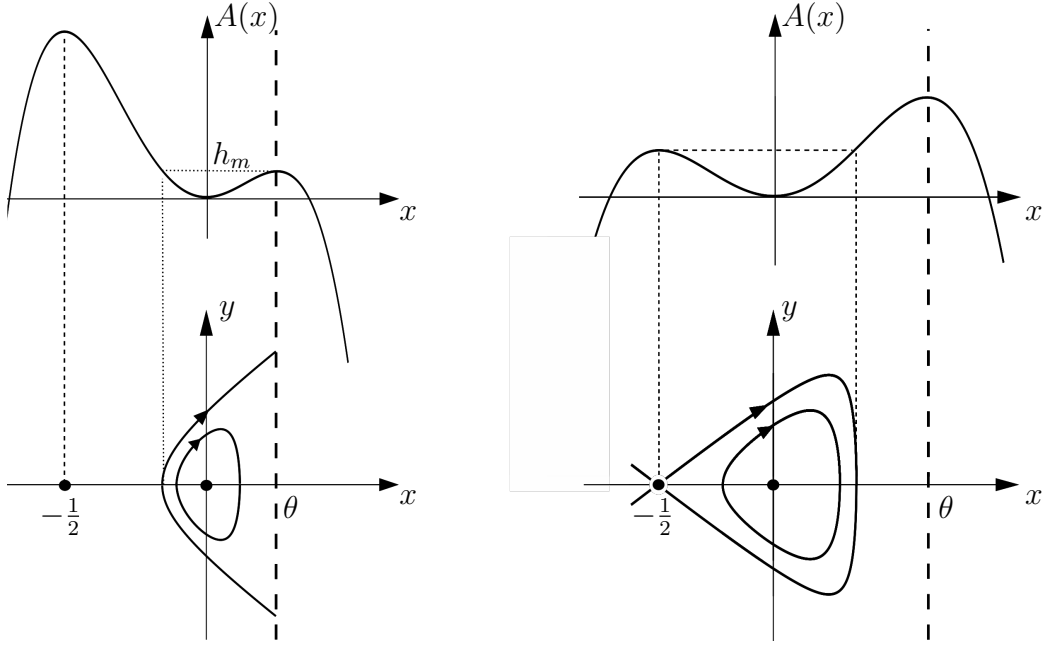

	\begin{lpic}{figure_theta(0.65,0.65)}
		\lbl{49,130;$A(x)$}
		\lbl{82,90;$x$}
		\lbl{82,32;$x$}
		\lbl{45,70;$y$}
		\lbl{58,30;$\theta$}
		\lbl{11,28;$-\frac{1}{2}$}
		\lbl{46,102;$h_m$}
		
		\lbl{165,130;$A(x)$}
		\lbl{210,90;$x$}
		\lbl{210,32;$x$}
		\lbl{160,70;$y$}
		\lbl{192,30;$\theta$}
		\lbl{132,28;$-\frac{1}{2}$}
	\end{lpic}
	
	\caption{A sketch of the period annulus $\mathscr P$ of the center at the origin of system \eqref{e-sys_x} for $\theta\in(0,\frac{1}{2})$ (left) and $\theta\in[\frac{1}{2},\infty)$ (right). }
	\label{fig-theta}
\end{figure}

Recall that a mapping $\sigma$ is said to be an involution if $\sigma\circ\sigma=\mbox{Id}$. The function  $A$ defines an \emph{involution} $\sigma$ satisfying $A=A\circ \sigma$. We find that 
 \begin{equation}\label{e-sig}
  A(x)-A(z)=- \frac{1}{6} (x-z)S(x,z), 
 \end{equation}
where 
$$
S(x,z):=3x^3 + (3z + 2 -4\theta)x^2 + (  3z^2 + 2z -4\theta z- 3\theta )x + 3z^3  + 2z^2 - 4\theta z^2 -3 \theta z,
$$
such that  $S\bigl(x,\sigma(x)\bigr)=0$. Let $(x_{\ell},x_r)$ be the projection onto the $x$-axis of the period annulus $\mathscr P$ around the center at the origin of the differential system \eqref{e-sys_x}. Given an analytic function $f$ on $(x_{\ell},x_r)\setminus\{0\}$ one can define its $\sigma$-\emph{balance} to be
\[
 \mathscr B_{\sigma}\bigl(f\bigr)(x)\!:=\frac{f(x)-f\bigl(\sigma(x)\bigr)}{2}.
\]
The number of zeros of the sigma balance of certain polynomials gives upper bounds for the number of critical points of the period function, see \cite{Vil2014}, as we will study below. The proof of the following auxiliary result is a straightforward computation of the first coefficients in the Taylor expansion of the period function $\ell(h)$ using standard techniques (see for example \cite{Gasull1997}).
 
\begin{lemma}
	\label{L-T0}
	The first two period constants of the period function $\ell(h)$ are given by 
	\begin{align*}
	&\Delta_1= \frac{\pi(4\theta+1)(5\theta-1)}{6\theta},\\
	&\Delta_2=-\frac{\pi}{288\theta^2}\left (-48\theta^6 - 144\theta^5 - 1808\theta^4 + 1152\theta^{5/2} + 1096\theta^3 + 741\theta^2 + 322\theta + 13\right ),
	\end{align*}
	such that the expansion of $\ell$ is given by $\ell(h) = 2\pi \sqrt{\theta}  + \Delta_1 h^2 + \Delta_2 h^4 + \mathcal O(h^5)$. 
\end{lemma}
We are now in position to prove monotonicity of the period function $\ell(h)$ for $\theta \geq \frac{1}{2}$. 

\begin{lemma}
	\label{theta>1/2}
If $\theta \geqslant \frac{1}{2},$ then the period function $\ell(h)$ is monotonically increasing.
\end{lemma}

\begin{proof}
For $\theta \geqslant \frac{1}{2},$ the projection of the period annulus on the $x$-axis is $(-\frac{1}{2},x_{r})$, where $A(x_r)=A(-\frac{1}{2})$. We will apply \cite[Theorem A and B]{Vil2014} and study the number of zeros of the sigma balance $\Bs (\ell_1)$ of $\ell_1$, where $\ell_1$ is defined in terms of $A, C$ and $K$, see \cite{Vil2014}, and takes the form
\begin{equation*}
     \ell_1(x)=\frac{\sqrt{2}}{6}\frac{(4\theta+1)(x+1)}{(2x+1)^3(x-\theta)}.
\end{equation*}
Note that since $\mathscr B_{\sigma}(f)\circ\sigma=-\mathscr B_{\sigma}(f)$ and $\sigma$ maps $(0,x_r)$  to $(x_{\ell},0)$,  we may for convenience study the latter interval, which in this case is $(-\frac{1}{2},0)$. 
We find that  $L\bigl(x,\ell_1(x)\bigr)\equiv 0$ with 
\begin{align*}
L(x,y)\!:=&\left( 4\,\theta+1 \right)  \left( x-y \right)  \left( -8\,x-8 \right) {y}^{3}
+ \left( 8\,x\theta-8\,{x}^{2}+8\,\theta-20\,x-12 \right) {y}^{2}\\
&+ \left( 8\,\theta\,{x}^{2}-8\,{x}^{3}+20\,x\theta-20\,{x}^{2}+12\,\theta-18\,x-6 \right) y\\
&+8\,\theta\,{x}^{2}-8\,{x}^{3}+12\,x\theta-12\,{x}^{2}+5\,\theta-6\,x-1.
\end{align*}
We find that 
$$
\mbox{Res}_z\bigl(L(x,z),L(y,z)\bigr)=8192(\theta+1)(4\theta+1)^8(x-y)^4  T
(x,y)^4,
$$ 
with $T$ a bivariate polynomial of degree $12$ in~$x$ and~$y$, which also depends polynomially on~$\theta$. Finally 
$$
\mathscr R(x)\!:=\mbox{Res}_y\bigl(S(x,y),T(x,y)\bigr)=(2x+1)^{12}(x-\theta)^4(\theta+1)^3 (4\theta+1)^4 R(x)^4,
$$ 
where~$R$ is a univariate polynomial of degree $10$ in $x$ depending polynomially on~$\theta$, and $S$ was defined in \eqref{e-sig}.

Let us denote by $\mathcal Z(\theta )$ the number of roots of $R$ on $(-\frac{1}{2},0)$ counted with multiplicities. We claim that $\mathcal Z(\theta)=0$ for all $\theta \geqslant\frac{1}{2}.$ For $\theta =\frac{1}{2}$ this can be easily verified by applying Sturm's method, see \cite[Theorem 5.6.2]{SB}. To prove it for $\theta >\frac{1}{2}$  note that
\begin{align*}
  &R(0)= (5\theta - 1)(\theta + 1)(64\theta^3 + 48\theta^2 + 21\theta + 1)  \\
\intertext{and}
  &R\left(-\frac{1}{2} \right)= \frac{9}{2}(2\theta - 1)(1 + 2\theta)^4\notag, 
\end{align*}
which do not vanish for $\theta>\frac{1}{2}$. The discriminant of $R$ with respect to~$x$, $\mbox{Disc}_x(R)$, is a polynomial of degree $70$ in $\theta$ with no real roots for $\theta>\frac{1}{2}$. Choosing one value of~$\theta>\frac{1}{2}$ and applying Sturm's method, we find that $\mathcal Z(\theta )=0$ for all $\theta \in (\frac{1}{2},+\infty)$ using \cite[Lemma 8.1]{GGG}. Therefore, $\mathscr R$ does not vanish on $(-\frac{1}{2},0)$ for any $\theta\geqslant\frac{1}{2}.$ In view of $(a)$ in \cite[Theorem A]{Vil2014} this implies that $\Bs\bigl(\ell_1\bigr)\neq 0$ on $(-\frac{1}{2},0).$ This proves the validity of the claim and hence, by applying \cite[Theorem A]{Vil2014} with $n=0,$ it follows that the period function is monotonous for $\theta\geqslant\frac{1}{2}.$ Finally, the result follows by noting that, thanks to Lemma \ref{L-T0}, the first period constant $\Delta_1$ is positive for $\theta\geqslant \frac{1}{2}$. \end{proof}

Now we study the period function $\ell(h)$ for $\theta \in(0, \frac{1}{2})$. 
The following lemma describes the behaviour of the period function near its outer boundary. 

\begin{lemma}
\label{L-T'lim}
If $\theta  \in (0,\frac{1}{2})$, then the period function $\ell(h)$ satisfies  $\lim\limits_{h\to h_m} \ell'(h)=-\infty$, where $h_m =A(\theta )$ is the energy level of the outer boundary of $\P$.
\end{lemma}

\begin{proof}
It was proven in \cite{Geyer2015b} that the derivative of the period function $\ell(h)$ can be written as 
 \begin{equation*}
   \ell'(h)=\frac{1}{h}\int_{\gamma_h}\!R(x)\,\frac{dx}{y}, 	
 \end{equation*}
 where 
 $$
 R\!=\frac{1}{2C}\!\left(\frac{K A}{A'}\right)'-\frac{K(AC)'}{4A'C^2}.
 $$  
Taking into account the respective definitions of these quantities, we find that 
 \begin{equation*}
 	R(x)=\frac{x(x+1)(4\theta+1)}{6(2x+1)^2(x-\theta)}. 
 \end{equation*}
 For $\theta\in(0,\frac{1}{2})$ and $h\in(0,h_m)$, we have that $h-A(x)=(x-x_h^-)(x-x_h^+)(x-x_h^\ell)(x-x_h^r)$, where $x_h^\ell<-\frac{1}{2}<x_\ell<x_h^-<0<x_h^+<\theta<x_h^r$, see 
 Figure~\ref{fig-sketch-roots}. In particular, the projection of $\gamma_h$ onto the $x$-axis is $[x_h^-,x_h^+]$. 
 
 \begin{figure}[htb!]
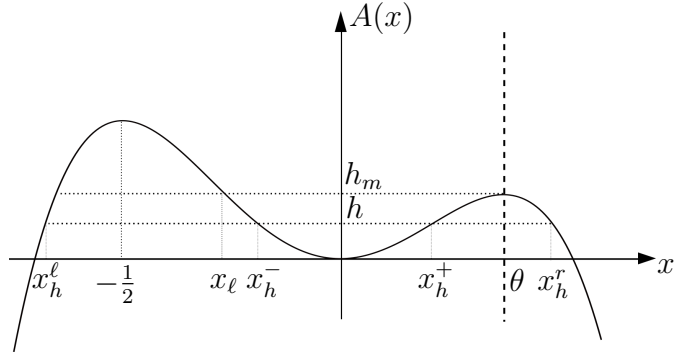

 	\begin{lpic}{roots(0.6,0.6)}
 		\lbl{87,75;$A(x)$}
 		\lbl{150,20;$x$}
 		\lbl{83,40;$h_m$}		
 		\lbl{81,33;$h$}		
 		
 		\lbl{13,17;$x_h^\ell$}				
 		\lbl{28,16;$-\frac{1}{2}$}		
 		\lbl{52,16;$x_\ell$}						
 		\lbl{61,17;$x_h^-$}				
 		\lbl{99,17;$x_h^+$}	
 		\lbl{117,17;$\theta$}	
 		\lbl{125,16;$x_h^r$}			
 	\end{lpic}
 	\caption{The distribution of  roots of the function $h-A(x)$ in the proof of Lemma \ref{L-T'lim} for $\theta\in(0,\frac{1}{2})$.}
 	\label{fig-sketch-roots}
 \end{figure}

We split the integral into two parts, $\ell'(h)= \frac{2}{h}\left (I_1(h)+I_2(h)\right )$, where 
 \begin{equation*}
 	I_1(h) = \int_{x_h^-}^0 f(x,h)\,dx\,\text{ and }
	I_2(h)=\int_0^{x_h^+}f(x,h)\,dx
 \end{equation*}
 with 
 \[
  f(x,h)=\frac{R(x)\sqrt{C(x)}}{\sqrt{h-A(x)}}
  =\frac{-(4\theta+1)x(x+1)}{6\sqrt{2}(2x+1)^2\sqrt{(x-x_h^-)(x-x_h^+)(x-x_h^\ell)(x-x_h^r)}}.
 \]
In order to study $I_1$, let us write $f(x,h)=\frac{g_1(x,h)}{\sqrt{x-x_h^-}}$, where 
 \[ 
  g_1(x,h)\!:=\frac{-(4\theta+1)x(x+1)}{6\sqrt{2}(2x+1)^2\sqrt{(x-x_h^+)(x-x_h^\ell)(x-x_h^r)}}.
 \]  
Note that $g_1$ is a continuous function on $(-\frac{1}{2},0]\!\times\! (0,h_m).$ Consequently there exists $M_1\in\R$ such that $M_1\!:=\sup\bigl\{g_1(x,h):(x,h)\in [x_r,0]\!\times\! [\frac{1}{2}h_m,h_m]\bigr\}.$ In addition, observe that $M_1>0$ for $-\frac{1}{2}<x<0$. Thus for $h\in (\frac{1}{2}h_m,h_m)$ we have that
  \begin{align*}
     I_1 (h)&= \int_{x_h^-}^0\frac{g_1(x,h)dx}{\sqrt{x-x_h^-}}\\
            &\leqslant M_1 \int_{x_h^-}^0\frac{dx}{\sqrt{x-x_h^-}}  \\
            &=2M_1\sqrt{-x_h^-}\\
            &<\sqrt{2}M_1.           
 \end{align*}
In order to study $I_2$ let us write $f(x,h)=\frac{g_2(x,h)}{\sqrt{(x-x_h^+)(x-x_h^r)}}$, where 
 \[ 
  g_2(x,h)\!:=\frac{-(4\theta+1)x(x+1)}{6\sqrt{2}(2x+1)^2\sqrt{(x-x_h^-)(x-x_h^\ell)}}.
 \]  
Since $g_2$ is continuous on $[0,\infty)\!\times\!(0,h_m),$ there exists $M_2\!:=\sup\bigl\{g_2(x,h):(x,h)\in [0,\theta]\!\times\! [\frac{1}{2}h_m,h_m]\bigr\}$ and we observe that $M_2<0$ for $0<x<\theta$. Consequently, if $h\in (\frac{1}{2}h_m,h_m),$ then
  \begin{align*} 	
     I_2 (h)&= \int^{x_h^+}_0\frac{g_2(x,h)dx}{\sqrt{(x-x_h^+)(x-x_h^r)}}\\
     &\leqslant M_2\int^{x_h^+}_0\frac{dx}{\sqrt{(x-x_h^+)(x-x_h^r)}} \\
                      &= M_2 \ln \left(\frac{\sqrt{x_h^r/x_h^+}+1}{\sqrt{x_h^r/x_h^+}-1}\right),
  \end{align*}
  where the upper bound diverges to $-\infty$ as $h \to h_m$ 
since $ M_2<0$ and $x_h^+, x_h^r \to \theta$ as $h \to h_m$. Since $\ell'(h)=\frac{2}{h}\bigl(I_1(h)+I_2(h)\bigr)$, the 
bound on $I_1(h)$ and the divergence of $I_2(h)$ as $h \to h_m$ imply the result.  
\end{proof}

We are now ready to prove  the monotonicity result of the period 
function $\ell(h)$ in the case that $\theta\in(0,\frac{1}{2})$.
\begin{lemma}
\label{theta<1/2}
For $\theta\in(0,\frac{1}{2})$ the period function $\ell(h)$
 \begin{enumerate}[$(a)$] 
  \item is monotonically decreasing for $\theta \in (0,\frac{1}{5}]$,
  \item has a unique critical period, which is a maximum, for 
           $\theta  \in (\frac{1}{5},\frac{1}{2})$.
\end{enumerate}           
\end{lemma}

\begin{proof}
For $\theta\in(0,\frac{1}{2})$ the projection of the period annulus on the $x$-axis is $(x_\ell,\theta)$, where $A(x_\ell)=A(\theta)=h_m$, see Figure \ref{fig-sketch-roots}. We proceed in exactly the same way as we did for $\theta>\frac{1}{2}$, i.e.~applying \cite[Theorems A and B]{Vil2014}. 

Let us now denote by $\mathcal Z(\theta )$ the number of roots of $R$ on $(0,\theta)$ counted with multiplicities and let $R$ be defined as in the proof of Lemma \ref{theta>1/2}. We find that $R(0) $ has a root at $\theta=\frac{1}{5}\in(0,\frac{1}{2})$, while 
$$
R(\theta)=(\theta + 1)^3(2\theta - 1)^3(1 + 2\theta)^4
$$ 
does not have a real root for $\theta\in(0,\frac{1}{2})$. The discriminant of $R$ with respect to $x,$ $\mbox{Disc}_{x}(R),$  is a polynomial of degree $70$  that has only one real root on $(0,\frac{1}{2})$ in $\theta=\frac{1}{5}$. Therefore, $\mathcal Z(\theta )$ is constant on $I_1:=(0,\frac{1}{5})$ and on $I_2:=(\frac{1}{5},\frac{1}{2})$. Choosing $\theta =\frac{1}{10}\in I_1$ we find that $R$ does not vanish on $(0,\frac{1}{10})$ and hence $\mathcal Z(\theta)=0$ for all $\theta \in(0,\frac{1}{5})$. For $\theta=\frac{1}{5}$, we find that $R\neq 0$ on $(0,\frac{1}{5})$ as  well by applying Sturm's method. Hence, it follows from \cite[Theorem A]{Vil2014} that $\Bs\bigl(\ell_1\bigr)\neq 0$ on $(0,\theta)$, and we may conclude that the period function is monotonous for $\theta\in (0,\frac{1}{5}]$. On the other hand, choosing $\theta =\frac{3}{10}\in I_2$ we find that $R$ vanishes once, which implies that the criterion in \cite[Theorem A]{Vil2014}  does not apply. Therefore, we move on to studying $\Bs\bigl(\ell_2\bigr)$, where 
\begin{equation*}
	\ell_2 = \frac{\sqrt 2}{36}\frac{(4\theta+1)P(x)}{(x-\theta)^3(2x+1)^5},
\end{equation*}
and $L(x,\ell_2(x))\equiv 0$ with $P$ and $L$  polynomials which we omit for the sake of brevity. As before, we compute $\mbox{Res}_z\bigl(L(x,z),L(y,z)\bigr)=T(x,y)$, with $T$ a bivariate polynomial which also depends polynomially on~$\theta$, and $\mathscr R(x)\!:=\mbox{Res}_y\bigl(S(x,y),T(x,y)\bigr)=12\sqrt{2}x(-x + \theta)^3(4\theta + 1)^3(2x + 1)^5 R(x)$, where~$R$ is a univariate polynomial of degree $30$ in $x$ depending polynomially on~$\theta$, and $S$ was defined in \eqref{e-sig}.

Let us denote  by $\mathcal Z(\theta )$ to be the number of roots of $R$ on $(0,\theta)$ counted with multiplicities. We claim that $\mathcal Z(\theta)=1$ for all $\theta \in(\frac{1}{5},\frac{1}{2}).$ We find that in that parameter interval, $R(0)R(\theta)\mbox{Disc}_x(R)\neq 0$ and hence $\mathcal Z(\theta)$ is constant. Evaluating in $\theta =\frac{3}{10}\in(\frac{1}{5},\frac{1}{2})$ and using Sturm's method we find that $R(x)$ has exactly one real root in $(0,\frac{3}{10})$ and hence $\mathcal Z(\theta)=1$ for all $\theta\in(\frac{1}{5},\frac{1}{2})$. In view of \cite[Theorem A]{Vil2014} for $i=2>1=n$ we may conclude that the number of critical  periods is at most $1$.

Recall from Lemma \ref{L-T'lim} that $\ell'(h)\to -\infty$ as $h$ tends to $h_m$ for all $\theta \in (0,\frac{1}{2})$. Since the first period constant $\Delta_1$ computed in Lemma \ref{L-T0} is negative for $\theta \in (0,\frac{1}{5})$ and positive for $\theta \in (\frac{1}{5},\frac{1}{2})$, we conclude that the period function $\ell(h)$ is monotonous decreasing near both endpoints of $(0,h_m)$ for all $\theta\in(0,\frac{1}{5})$, while  it is increasing near $h=0$ and decreasing near $h=h_m$ for $\theta\in(\frac{1}{5},\frac{1}{2})$. For $\theta=\frac{1}{5}$ we have that $\Delta_1=0$ and $\Delta_2<0$, and hence the period function is decreasing near the endpoint $h=0$. Taking into account the upper bounds derived above, we may conclude that the period function $\ell(h)$ is monotonous decreasing for $\theta \in  (0,\frac{1}{5}]$ and it has a unique critical period which is a maximum for $\theta\in(\frac{1}{5},\frac{1}{2})$. 
\end{proof}

\begin{remark}
	For the sake of completeness we remark that the limit value of the integral defining the period function at the right endpoint of its interval of definition is given by
	\begin{equation*}
		\ell_{h_m}=\frac{1}{2}\ln \left ( \frac{(\theta +1)(1-2\theta)}{4\theta + 1 +3 \sqrt{\theta(1+2\theta)}}\right ),
	\end{equation*}
which is positive and finite on $(0,\frac{1}{2})$, and can be obtained by standard techniques. 
\end{remark}

We finish this section with the proof of Lemma \ref{lem-nonmonotonicity}. 

\begin{proof}[Proof of Lemma \ref{lem-nonmonotonicity}]
The smooth periodic solutions of the second-order equation (\ref{CHode}) are periodic orbits of system \eqref{e-sys_x}, which are parametrized by $h\in(0,h_m)$ and  whose periods are assigned by the period function $\ell(h)$.  A straightforward computation shows that $\mathfrak L(a,b) = \ell(h)$ and 
$$
a=-\left(\frac{\sqrt{\Delta}-3c}{4}\right)^2\left(2h\frac{\Delta}{\theta}+\frac{\sqrt{\Delta}-3c}{4}\left(\frac{\sqrt{\Delta}-3c}{4}+c\right )\right).
$$
Therefore, $\frac{d a}{dh}<0$ and  so for fixed $b\in (-c^2,\frac{1}{8}c^2)$ and $c>0$ we have that 
$$
{\rm sign}(\partial_a \mathfrak L(a,b))=-{\rm sign}(\ell'(h)), 
$$
which means that the monotonicity properties of $\ell(h)$ imply those of $\mathfrak L(a,b)$. More precisely, in view of the definition of $\theta$, the parameter regime $\theta \geq \frac{1}{2}$ corresponds to values $b\in[0,\frac{1}{8}c^2)$ for which $\partial_a\mathfrak L(a,b)<0$ in view of Lemma \ref{theta>1/2}. On the other hand, 
the value $\theta = \frac{1}{5}$ corresponds to $b = -\frac{2}{9} c^2$ and 
we infer from Lemma \ref{theta<1/2} that $\partial_a\mathfrak L(a,b)>0$ for $b\in(-c^2,-\frac{2}{9}c^2)$  whereas $\mathfrak L(a,b)$ has a unique critical point in $a$, which is a maximum, for $b\in(-\frac{2}{9}c^2,\frac{1}{8}c^2)$. This concludes the proof.
\end{proof}

\begin{figure}[htb!]
	\includegraphics[width=5cm,height=4cm]{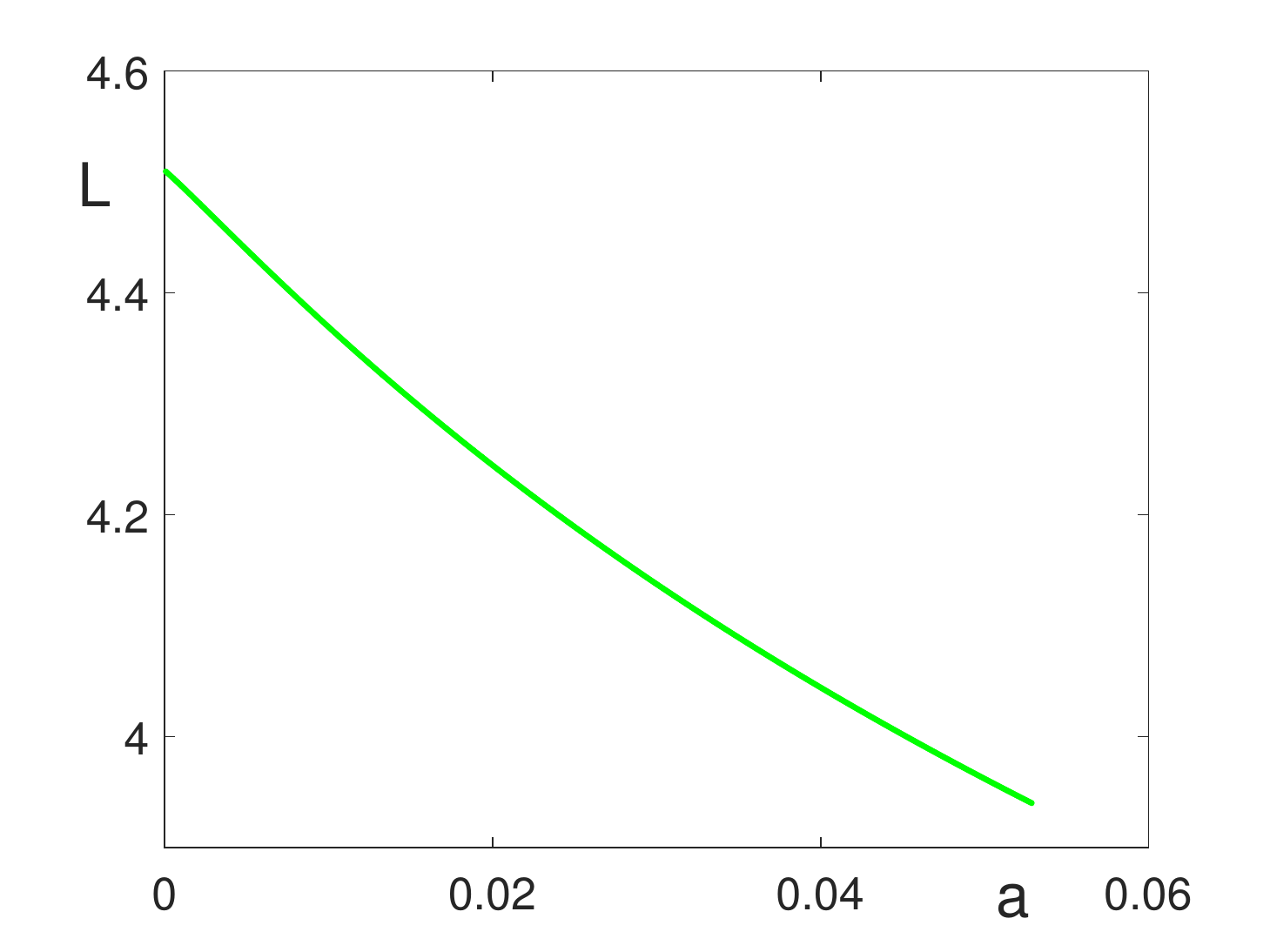} 
	\includegraphics[width=5cm,height=4cm]{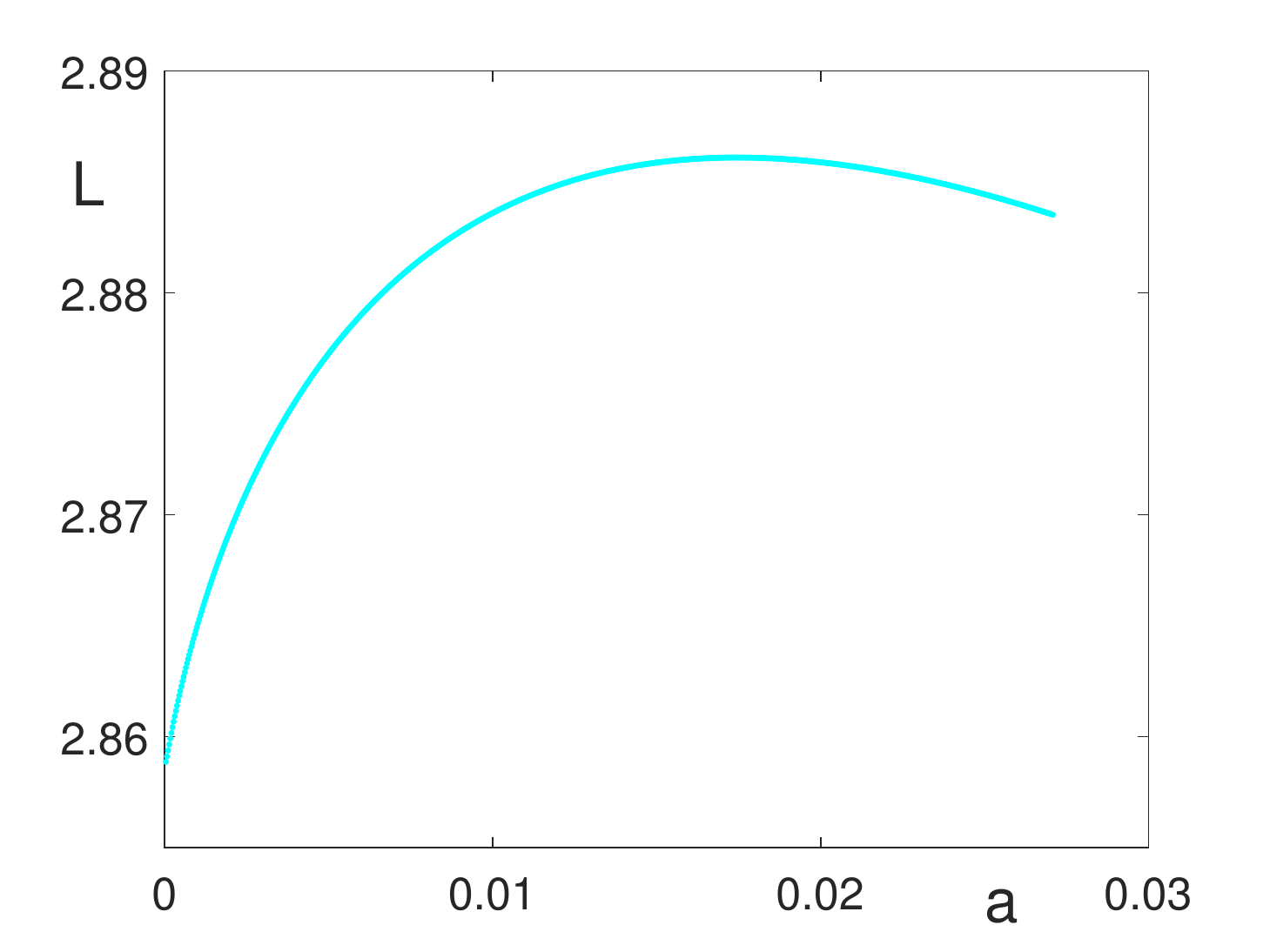} 
	\includegraphics[width=5cm,height=4cm]{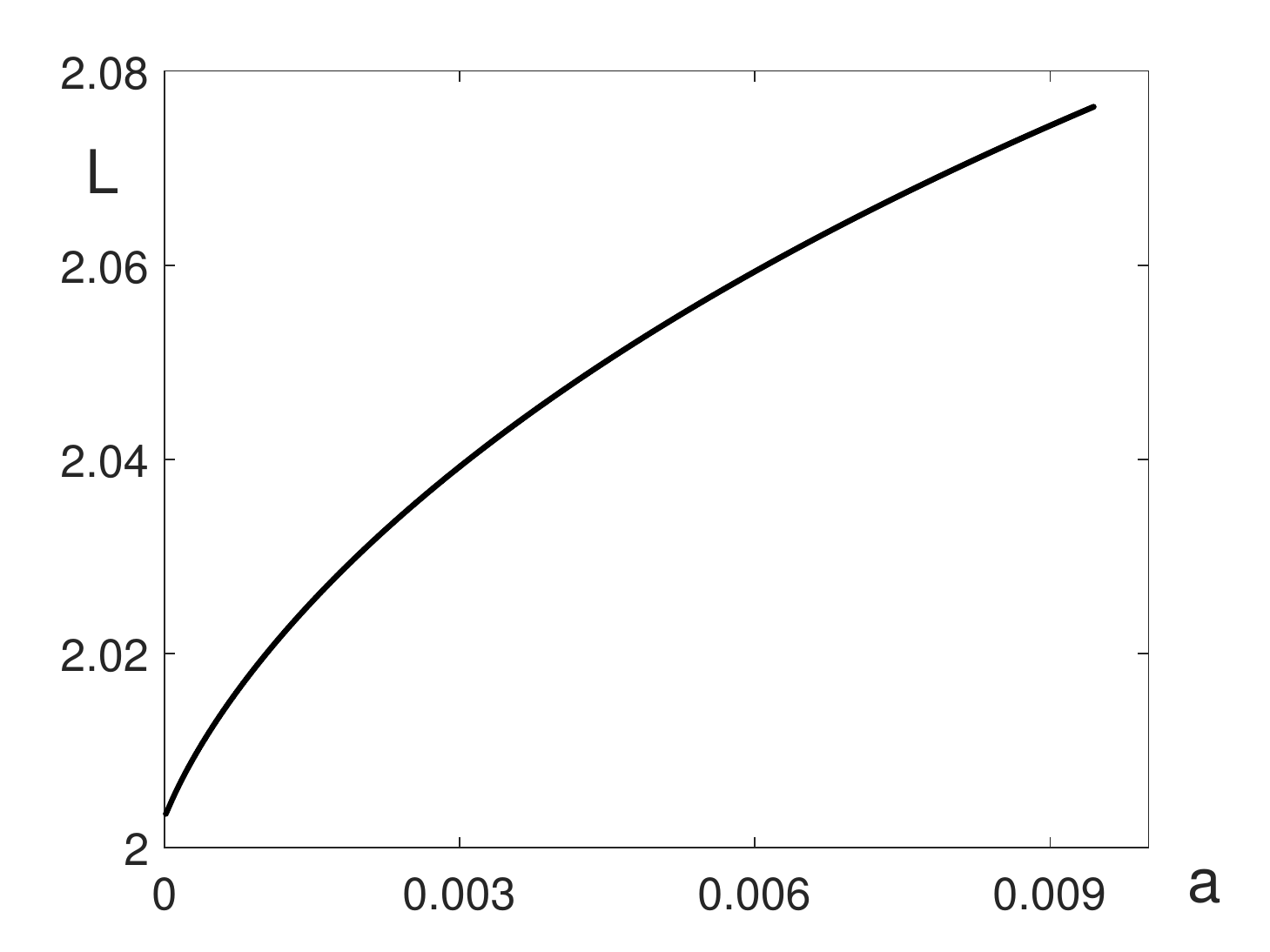} 
	\caption{The period function $\mathfrak{L}(a,b)$ versus $a$ for $c = 1$ and  three values of $b$: 
		$b = 0$ (left), $b = -0.2$ (middle), and $b = -0.4$ (right).}
	\label{fig-periodfunction}
\end{figure}

Figure \ref{fig-periodfunction} illustrates the result of Lemma \ref{lem-nonmonotonicity} for $c = 1$. The period function $\mathfrak{L}(a,b)$ is monotonically decreasing in $a$ for $b = 0$, is non-monotone in $a$ with a single maximum for $b = -0.2$, and is monotonically increasing in $a$ for $b = -0.4$. 
The range of $a$ values depends on the values of $b$ as is clear from Figure \ref{fig-domain}.  Note that the colors do not correspond to the colors of Figure \ref{fig-dependence}, where the values of $L = \mathfrak{L}(a,b)$ are fixed.

\section{Spectral properties of the Hessian operator $\mathcal{L}$}
\label{sec-3}

Here we shall consider the spectral properties of the Hessian operator $\mathcal{L}$ given by (\ref{hill}). 
Since $\mathcal{L}: L^2_{\rm per} \to L^2_{\rm per}$ is self-adjoint, its spectrum  $\sigma(\mathcal{L})$ consists of the absolutely continuous part, denoted by $\sigma_c(\mathcal{L})$, and the point spectrum, denoted by $\sigma_p(\mathcal{L})$. 
Since $c - \phi$ is a bounded multiplicative operator in $L^2_{\rm per}$
and $-3c (4 - \partial_x^2)^{-1}$ is a compact operator in $L^2_{\rm per}$, 
Kato's theorem \cite{Kato} implies that 
$$
\sigma_c(\mathcal{L}) = \sigma(c - \phi) = {\rm Range}(c-\phi)
$$
Since $c - \phi > 0$ by Lemma \ref{lem-trav}, there exists 
$$
\lambda_0 := c - \max\limits_{x \in \mathbb{T}_L} \phi(x) > 0
$$ 
such that $\sigma_p(\mathcal{L})$ admits finitely many eigenvalues of finite multiplicities below $\lambda_0$. 

The following lemma gives an efficient technique to count the negative and zero eigenvalues of $\mathcal{L}$. It is an analogue of the Birman--Schwinger principle used in quantum mechanics \cite[Section 5.6]{GS-book}. A similar criterion was developed in our previous work \cite{Geyer}.

\begin{lemma}
	\label{lem-BS} 
	For every $\lambda \in (-\infty,\lambda_0)$ with $\lambda_0>0$, let the Schr\"{o}dinger operator $\mathcal{K}(\lambda) : H^2_{\rm per} \subset L^2_{\rm per} \to L^2_{\rm per}$ be defined by 
\begin{equation}
\label{hill-M}
	\mathcal{K}(\lambda) := - \partial_x^2 + \frac{c - 4 \phi - 4 \lambda}{c - \phi - \lambda}. 
\end{equation}
	Then, we have 
	\begin{align}
\#\{ \lambda < 0 : \;\; \mathcal{L} w = \lambda w, \;\; w \in L^2_{\rm per}\} 
&= \#\{ \mu < 0 : \;\; \mathcal{K}(0) v = \mu v, \;\; v \in H^2_{\rm per}\},
	\label{equality-2}
	\end{align}
where $\#\{ \cdot \}$ denotes the number of eigenvalues, taking into  account their multiplicities.
\end{lemma}

\begin{proof}
The spectral problem $\mathcal{L} w = \lambda w$ with $w \in L^2_{\rm per}$ 
can be rewritten in the variable $v := (4 - \partial_x^2)^{-1} w$ as the spectral problem $\mathcal{K}(\lambda) v = 0$ with $v \in H^2_{\rm per}$. Since the operator $4 - \partial_x^2$ is invertible with a bounded inverse in $L^2_{\rm per}$, the correspondence $v = (4 - \partial_x^2)^{-1} w$ implies that if $\lambda < 0$ is an eigenvalue of $\mathcal L$, then $\mathcal K(\lambda)$ admits a zero eigenvalue of the same multiplicity. Because of the compact embedding of $H^2_{\rm per}$ into $L^2_{\rm per}$, we have 
$$
\sigma(\mathcal{K}(\lambda)) = \sigma_p(\mathcal{K}(\lambda)) \quad 
\mbox{\rm for} \;\; \lambda \in (-\infty,\lambda_0),
$$ 
that is, the spectrum of $\mathcal{K}(\lambda)$ consists of eigenvalues as long as 
$$
A(x,\lambda) := \frac{c - 4 \phi(x) - 4 \lambda}{c - \phi(x) - \lambda}
$$
is bounded in $x$. Since
$$
\partial_{\lambda} A(x,\lambda) = -\frac{3c}{(c-\phi(x) - \lambda)^2} < 0,
$$
the eigenvalues of $\mathcal{K}(\lambda)$ are monotonically decreasing functions of $\lambda$. Since 
$$
\lim_{\lambda \to -\infty} A(x,\lambda) = 4,
$$
there exists $\lambda_{\infty} \in (-\infty,0)$ such that $A(x,\lambda)>0$ for all $x \in \mathbb{R}$ and $\lambda <\lambda_{\infty}$, and hence 
$\sigma_p(\mathcal{K}(\lambda)) > 0$ for $\lambda \in (-\infty,\lambda_{\infty})$.
Each eigenvalue of $K(\lambda)$, say $\mu(\lambda)$, is decreasing and positive for large negative $\lambda$, and therefore crosses the horizontal axis at most once in $(-\infty, 0)$. If there exists $\lambda \in (-\infty, 0)$ such that $\mu(\lambda)=0$, then $\mu(0)<0$, i.e. it corresponds to a negative eigenvalue of $K(0)$. Therefore, the number of negative eigenvalues of $K(0)$ equals the number of $\lambda$ for which $K(\lambda)v=0$. In view of the previous equality with the number of negative eigenvalues of $\mathcal{L}$, this proves the equality (\ref{equality-2}).
\end{proof}

\begin{remark}
	\label{remark-zero}
	Because $\mathcal{L} \phi' = 0$, we have $\mathcal{K}(0) \nu' = 0$ so that  $0$ is an eigenvalue of $\mathcal{K}(0)$.
\end{remark}

Figure \ref{fig-eigenvalues} illustrates the criterion in Lemma \ref{lem-BS} 
with numerical approximations of the eigenvalues of $\mathcal{K}(\lambda)$ in $L^2_{\rm per}$ versus $\lambda$ for two different values of $(a,b)$ with $c = 1$. The left panel 
corresponds to the choice $(a,b) = (0.04,0)$ above the curve $a = a_0(b)$ shown on Figure \ref{fig-domain}. 
Only the first eigenvalue of $\mathcal{K}(\lambda)$ crosses the zero level (dotted line) in $(-\infty,0)$, whereas 
the second eigenvalue crosses the zero level at $\lambda = 0$. 
The right panel corresponds to the choice 
$(a,b) = (0.001,-0.3)$ below the curve 
$a = a_0(b)$ shown on Figure \ref{fig-domain}. The first two eigenvalues of $\mathcal{K}(\lambda)$ cross the zero level in $(-\infty,0)$ and the third eigenvalue, which is close to the second eigenvalue, crosses the zero level at $\lambda = 0$. 
The zero eigenvalue of $\mathcal{K}(0)$ exists in both cases, in accordance with Remark \ref{remark-zero}.

\begin{figure}[htb!]
		\includegraphics[width=7cm,height=6cm]{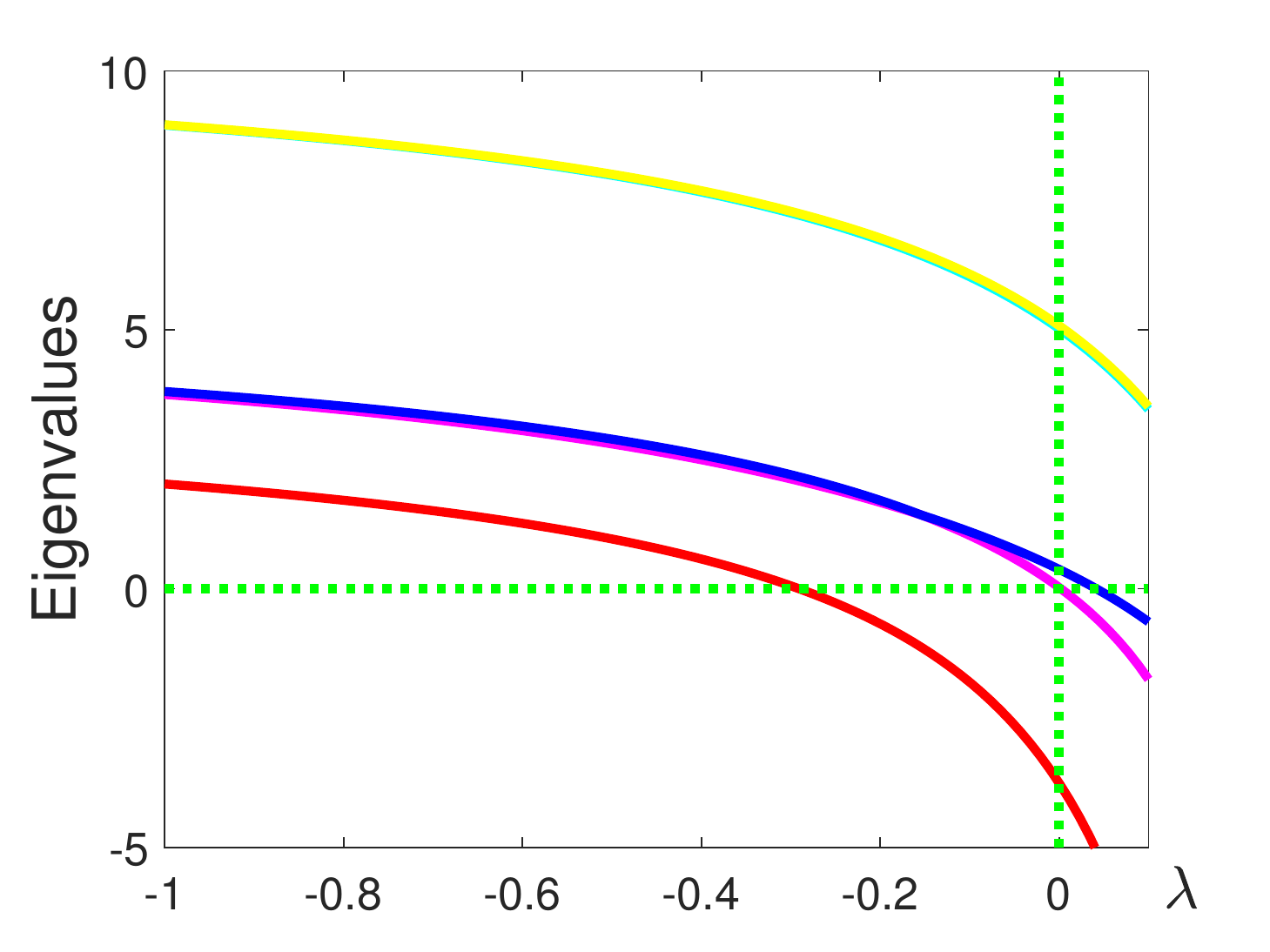}
		\includegraphics[width=7cm,height=6cm]{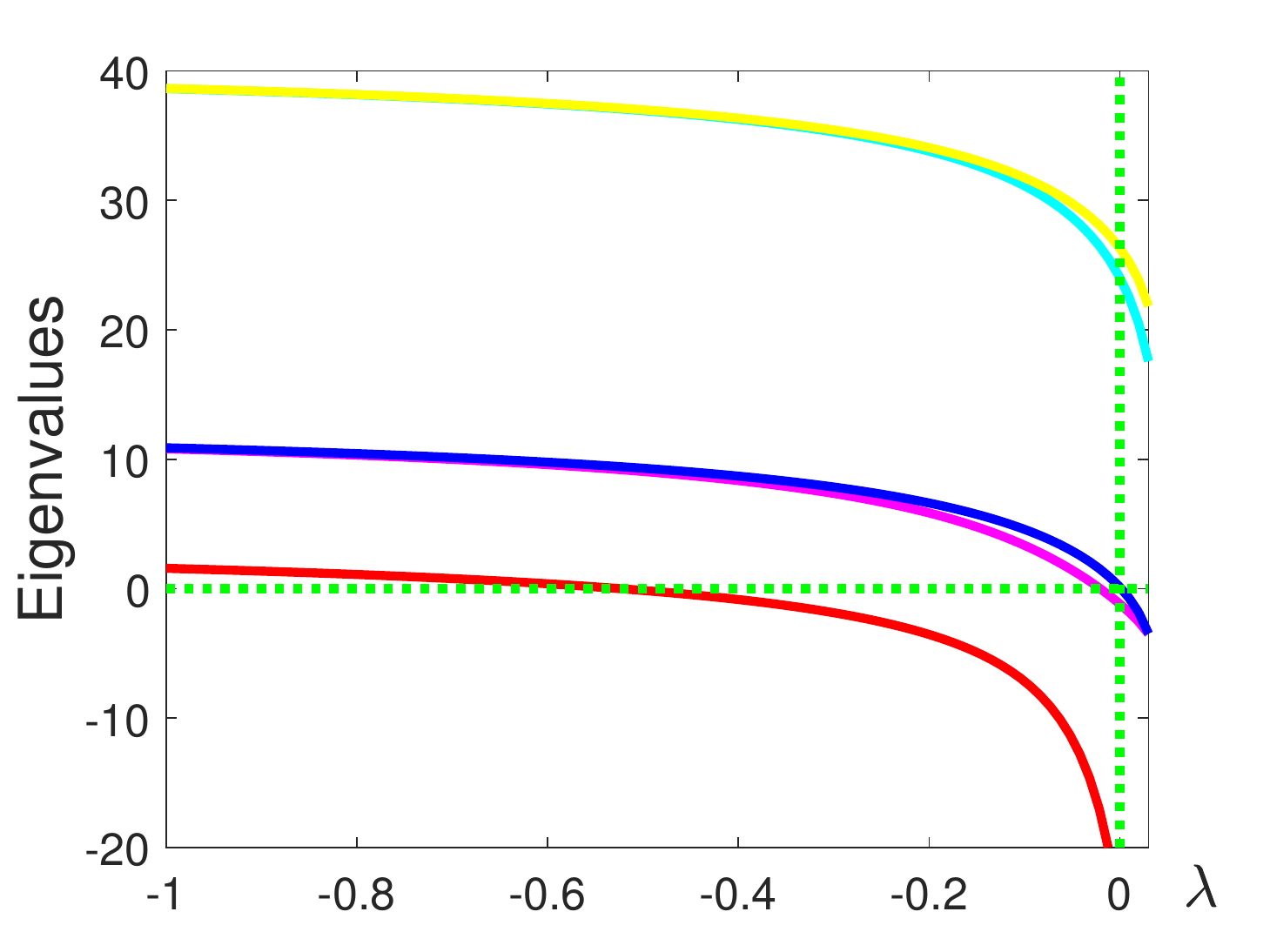}
	\caption{The lowest five eigenvalues of $\mathcal{K}(\lambda)$ versus $\lambda$ for $a = 0.04$, $b = 0$ (left) and $a = 0.001$, $b = -0.3$ (right) with $c = 1$. Eigenvalues are strictly decreasing in $\lambda$.} 
	\label{fig-eigenvalues}
\end{figure}

The next result uses the criterion in Lemma \ref{lem-BS} to relate the number of 
negative eigenvalues and the multiplicity of the zero eigenvalue of $\mathcal{L}$ 
with the period function $\mathfrak{L}(a,b)$ defined in (\ref{period-L-function}). Similar considerations can be found in \cite{GMNP}, \cite[Lemma 4.2]{J}, and \cite[Theorem 3.1]{Neves}.

\begin{lemma}
	\label{lem-Floquet}
	The linearized operator $\mathcal{L} : L^2_{\rm per} \to L^2_{\rm per}$ given by (\ref{hill}) admits
	\begin{itemize}
		\item two negative eigenvalues and a simple zero eigenvalue  if $\partial_a \mathfrak{L} > 0$;
		\item a simple negative eigenvalue and a double zero eigenvalue if $\partial_a \mathfrak{L} = 0$;
		\item a simple negative eigenvalue and a simple zero eigenvalue if $\partial_a \mathfrak{L} < 0$,
	\end{itemize}
	where $\mathfrak{L}(a,b)$ is given by (\ref{period-L-function}),
	and the rest of its spectrum in $L^2_{\rm per}$ is strictly positive. 
\end{lemma}

\begin{proof}
	By Lemma \ref{lem-BS}, we need to control the negative and zero eigenvalues 
	of the linear operator $\mathcal{K}(0) : H^2_{\rm per} \subset L^2_{\rm per} \to L^2_{\rm per}$ given by (\ref{hill-M}). Using the change of variables $w= (4 - \partial_x^2)v$, the second-order differential equation $\mathcal{K}(0) v = 0$ can be written as
	$$
	c(v - v'') - \phi w = 0.
	$$ 
	This equation has the two solutions $v_1 = \nu'$ and $v_2 = \partial_a \nu$, which follows by differentiating (\ref{nu-phi-equation}) in $x$ and $a$ since $c$ and $d = b/4$ are independent of $x$ and $a$.
	
	Let $\{ y_1,y_2\}$ be the fundamental set of solutions associated to the equation $\mathcal{K}(0) v=0$ in $H^2(0,L)$ such that 
	\begin{equation}
	\label{fund-set}
	\left\{ \begin{array}{l} y_1(0) = 1, \\
	y_1'(0) = 0, \end{array} \right. \qquad 
	\left\{ \begin{array}{l} y_2(0) = 0. \\
	y_2'(0) = 1, \end{array} \right.
	\end{equation}
We set $\phi(0) = \phi(L) = \phi_+$ and $\phi'(0) = \phi'(L) = 0$, 
where $\phi_+$ is the turning point for the maximum of $\phi$ in $x$ 
satisfying the equation 
\begin{equation}
\label{roots-turning-points}
(c - \phi_{\pm})^2 (b + \phi_{\pm}^2) = a,
\end{equation}
in view of the first-order invariant \eqref{quadra}. 
It follows from (\ref{v-phi}) that we can define 
\begin{equation}
\label{nu-turning-points}
\nu_{\pm} := \frac{1}{3} \phi_{\pm} - \frac{1}{6c} \phi_{\pm}^2 - \frac{b}{12c}
\end{equation}
as the corresponding turning points of $\nu = (4 - \partial_x^2)^{-1} \phi$. 
We compute from (\ref{nu-phi-equation}) and (\ref{nu-turning-points}) that 
$$
\nu''(0) = \frac{1}{3c} (c \phi_+ - 2 \phi_+^2 - b) = \frac{1}{3c} (c- \phi_+) \phi''(0)
$$
and 
$$
\partial_a \nu_+ = \frac{1}{3c} (c - \phi_+) \partial_a \phi_+,
$$
which are both nonzero since $c - \phi_+ > 0$, 
$\partial_a \phi_+ \neq 0$, and $\phi''(0) \neq 0$. Moreover, 
differentiating  (\ref{roots-turning-points}) in $a$ yields 
$$
2 (c-\phi_+) (c \phi_+- 2 \phi_+^2 - b) \partial_a \phi_+ = 1,
$$ 
from which we obtain
$$
\phi''(0) \partial_a \phi_+ =  \frac{1}{2(c-\phi_+)^2} > 0.
$$
Due to the normalization (\ref{fund-set}), we can then define
	\begin{equation*}
	y_1(x) := \frac{\partial_a \nu(x)}{\partial_a \nu_+}, \quad 
	y_2(x) := \frac{\nu'(x)}{\nu''(0)},
	\end{equation*}
and obtain $y_1(L) = y_1(0) = 1$, $y_1'(0)=0$, and 
	$$
	y_1'(L) = -\frac{\partial_a \mathfrak{L}}{\partial_a \nu_+} \nu''(0) = 
	-\frac{\partial_a \mathfrak{L}}{\partial_a \phi_+} \phi''(0) = 
		-\frac{\partial_a \mathfrak{L}}{2 (c-\phi_+)^2 (\partial_a \phi_+)^2},
	$$
where we have differentiated $\nu'(L)=0$ with respect to $a$ and used that $L = \mathcal L(a,b)$. On the other hand, $y_2(L) = y_2(0) = 0$ and $y_2'(L) = y_2'(0) = 1$. If we denote $\theta := y_1'(L)$, then $y_1(x+L) = y_1(x) + \theta y_2(x)$. Since the sign of $\theta$ is opposite to that of $\partial_a \mathfrak{L}$, the assertion follows from \cite[Proposition 1]{GMNP}.
\end{proof}

\begin{remark}
	In the case of  smooth solitary waves on a constant background, 
	 the Schr\"{o}dinger operator operator $\mathcal{K}(\lambda) : H^2(\mathbb{R}) \subset L^2(\mathbb{R}) \to L^2(\mathbb{R})$ for $\lambda \in (-\infty,\lambda_0)$ admits a finite number of simple isolated eigenvalues and an absolutely continuous spectrum 	located in $[\mu_{\infty},\infty)$, where $\mu_{\infty} := \lim\limits_{|x| \to \infty} A(x,\lambda)$. Since $\phi(x) \to \phi_1$ as $|x| \to \infty$ on the top boundary of the region of Lemma \ref{lem-trav}, we have 
	$$
	\lim_{|x| \to \infty} A(x,0) = \frac{c - 4 \phi_1}{c - \phi_1} > 0
	$$
	and $\mathcal{K}(0) \nu' = 0$ with $\nu' \in H^2(\mathbb{R})$. By Sturm's nodal theorem, we have 
	$$
	\#\{ \mu < 0 : \;\; \mathcal{K}(0) v = \mu v, \;\; v \in H^2(\mathbb{R}) \} = 1
	$$ 
	so that 
	$$
	\#\{ \lambda < 0 : \;\; \mathcal{L} w = \lambda w, \;\; w \in L^2(\mathbb{R})\} = 1
	$$ 
	by the criterion in Lemma \ref{lem-BS}. Thus, $\mathcal{L}: L^2(\mathbb{R}) \to L^2(\mathbb{R})$ has a simple negative eigenvalue and a simple zero eigenvalue in the case of  smooth solitary waves. This yields 
	a much simpler argument compared to the theory developed in \cite{Liu-21}. 
\end{remark}

We use the monotonicity properties of the period function in Lemma \ref{lem-nonmonotonicity} and the criterion in Lemma \ref{lem-Floquet} in order to prove the last assertion of Theorem \ref{theorem-existence} 
stated as the following corollary.

\begin{corollary}
	\label{lem-degenerate}
	For a fixed $c > 0$, there exists a smooth curve $a = a_0(b)$ 
	for $b \in (-\frac{2}{9}c^2,0)$ inside the existence region of smooth periodic waves in Lemma \ref{lem-trav} such that the linear operator $\mathcal{L}$ in $L^2_{\rm per}$ has only one simple negative eigenvalue above the curve and two simple negative eigenvalues (or a double negative eigenvalue) below the curve, the rest of its spectrum  for $a \neq a_0(b)$ includes a simple zero eigenvalue and a strictly positive spectrum bounded away from zero. Along the curve $a = a_0(b)$,  the linear operator $\mathcal{L}$ in $L^2_{\rm per}$ has only one simple negative eigenvalue, a double zero eigenvalue, and the rest of its spectrum is strictly positive. 
\end{corollary}

\begin{proof}
	Let $n(\mathcal{L})$ denote the number of negative eigenvalues of $\mathcal{L}$, taking into account their multiplicities. By Lemma \ref{lem-nonmonotonicity}, $\partial_a \mathfrak{L} > 0$ for every $a$ if $b \in (-c^2,-\frac{2}{9}c^2]$ so that $n(\mathcal{L}) = 2$ by Lemma \ref{lem-Floquet}. 
	Similarly, $\partial_a \mathfrak{L} < 0$ for every $a$ if $b \in [0,\frac{1}{8} c^2)$ so that $n(\mathcal{L}) = 1$. 
	For $b \in (-\frac{2}{9}c^2,0)$, there exists exactly one $a = a_0(b)$ for which the mapping $a \mapsto \mathcal{L}(a,b)$ has the maximum point.
	This curve is shown on Figure \ref{fig-domain}. Hence, 
	$\partial_a \mathfrak{L}> 0$ for $a < a_0(b)$ with $n(\mathcal{L}) = 2$ 
	and $\partial_a \mathfrak{L}< 0$ for $a > a_0(b)$ with $n(\mathcal{L}) = 1$. Combining the results in these three regions, we conclude that $n(\mathcal{L}) = 1$ above the curve and $n(\mathcal{L}) = 2$ below the curve inside the existence region. Along the curve $a = a_0(b)$, $\partial_a \mathfrak{L} = 0$ so that $n(\mathcal{L}) = 1$ and the zero eigenvalue of $\mathcal{L}$ is double.
\end{proof}

\begin{proof}
	The proof of Theorem \ref{theorem-existence} is complete with the results of Lemma \ref{lem-trav} Lemma \ref{theorem-increasing}, and Corollary \ref{lem-degenerate}.
\end{proof}


\section{Energy stability criterion}
\label{sec-5}

To study the stability of the smooth periodic traveling waves with the profile $\phi$ with respect to co-periodic perturbations, we consider the decomposition $$
u(t,x) = \phi(x-ct) + w(t,x-ct).
$$ 
When this is substituted into the DP equation (\ref{DP}) and quadratic terms in $w$ are neglected, we obtain the linearized equation in the form 
$$
w_t - w_{txx} - c w_x + c w_{xxx} + 4 \phi w_x + 4 w \phi' = 
3 \phi' w_{xx} + 3 w_x \phi'' + \phi w_{xxx} + w \phi''',
$$
where $x$ stands for the traveling wave coordinate $x - ct$. 
The linearized equation can be written in the Hamiltonian form 
\begin{equation}
\label{DPlin}
w_t = - J \mathcal{L} w,
\end{equation}
where $J$ is the same as in (\ref{sympl-1}) and $\mathcal{L}$ is the same as in (\ref{hill}). Indeed, the equivalence of the linearized equations follows from 
the relation 
\begin{align*}
& \partial_x (4 - \partial_x^2) \left[ (c-\phi) w - 3 c (4 - \partial_x^2)^{-1} w \right] \\
&= (c - 4 \phi) w_x - (c-\phi) w_{xxx} - 4 \phi' w + 3 \phi' w_{xx} + 3 \phi'' w_x + \phi''' w.
\end{align*}
Linearization of the mass and energy functionals (\ref{Mu}) and (\ref{Fu}) 
at the traveling wave with the profile $\phi$ by using the co-periodic perturbation with the profile $w$ yields the constrained subspace of $L^2_{\rm per}$ of the form
\begin{equation}
\label{orth-cond}
X_0 := \left\{ w \in L^2_{\rm per} : \quad 
\langle 1, w \rangle = 0, \quad \langle \phi^2, w \rangle = 0 \right\}.
\end{equation} 
The following lemma shows that the two constraints are invariant in the time evolution of the linearized equation (\ref{DPlin}). 

\begin{lemma}
	\label{lem-constraints}
	Let $w \in C(\mathbb{R},H^s_{\rm per}) \cap C^1(\mathbb{R},H^{s-1}_{\rm per})$ be the global solution to the linearized equation (\ref{DPlin}) with  $s > \frac{3}{2}$ for  initial data $w(0,\cdot) = w_0 \in H^s_{\rm per}$. If $w_0 \in X_0$, then $w(t,\cdot) \in X_0$ for every $t \in \mathbb{R}$.
\end{lemma}

\begin{proof}
Since $J$ is skew-adjoint  and $J 1 = 0$, we obtain
$$
\frac{d}{dt} \langle 1, w \rangle = -\langle 1, J \mathcal{L}w \rangle = \langle J 1, \mathcal{L} w \rangle = 0.
$$
Similarly, since $\mathcal{L}^* = \mathcal{L}$, $J \phi^2 = -2c \phi'$, and $\mathcal{L}\phi' = 0$, we obtain
$$
\frac{d}{dt} \langle \phi^2, w \rangle = -\langle \phi^2, J \mathcal{L}w \rangle = \langle J \phi^2, \mathcal{L} w \rangle = -2c \langle \phi', \mathcal{L}w \rangle = -2c \langle \mathcal{L}\phi', w \rangle = 0.
$$
It follows from the invariance of the two constraints under the time evolution 
of the linearized equation (\ref{DPlin}) that if $w_0 \in X_0$, then $w(t,\cdot) \in X_0$ for every $t \in \mathbb{R}$.
\end{proof}

Formal differentiation of the second-order equation (\ref{nu-phi-equation}) 
with $d = b/4$ in  $b$ and $c$ yields
\begin{eqnarray}
\label{der-phi}
\mathcal{L} \partial_b \phi = \frac{1}{4}, \quad 
\mathcal{L}\partial_c \phi = -\frac{b}{4c} - \frac{\phi^2}{2c}.
\end{eqnarray}
The relations (\ref{der-phi}) allow us to characterize $1 \in {\rm Range}(\mathcal{L})$ 
and $\phi^2 \in {\rm Range}(\mathcal{L})$ in $L^2_{\rm per}$ provided that we can take derivatives in $b$ and $c$ of the family of periodic waves with the profile $\phi \in H^{\infty}_{\rm per}$ along a curve with fixed period $L = \mathcal{L}(a,b)$. 

The following lemma uses the fact that the period function is monotone in $b$, see Lemma \ref{theorem-increasing}, to guarantee the existence  of a unique curve in the $(a,b)$ parameter space for which solutions $\phi$ have a fixed period $L$ for every $L \in (0,\infty)$. 

\begin{lemma}
	\label{lemma-fixed-period}
	Fix $c > 0$ and $L > 0$. There exists a $C^1$ mapping $a \mapsto b = \mathcal{B}_L(a)$ for $a \in (0,a_L)$ with  some $a_L \in (0,\frac{27}{256} c^4)$ and a $C^1$ mapping $a \mapsto \phi = \Phi_L(\cdot,a) \in H^{\infty}_{\rm per}$ of smooth $L$-periodic solutions along the curve $b = \mathcal{B}_L(a)$.
\end{lemma} 

\begin{proof}
It follows from (\ref{dependence-b-L}) that the mapping $b \mapsto L = \mathfrak{L}(0,b) \in (0,\infty)$ is one-to-one and 
onto at the boundary $a = 0$, where $b \in(-c^2,0)$. The limiting $L$-periodic 
wave has a peaked profile $\phi$ on the boundary $a = 0$.

Similarly, at the boundary $b = b_-(a)$, the limiting $L$-periodic wave corresponds to the constant wave $\phi = \phi_2$ and the period $L$ is found from the linearization of the second-order equation (\ref{second-order}) at $\phi = \phi_2$. A simple computation for $\varphi := \phi - \phi_2$ yields the linearized equation in the form 
\begin{equation*}
\varphi'' + \left ( \frac{3a}{(c-\phi_2)^4} - 1 \right )\varphi = 0.
\end{equation*}
Since $a = \phi_2 (c-\phi_2)^3$ on the boundary $b = b_-(a)$, it follows that 
the mapping 
\begin{equation}
\label{lin-freq}
\phi_2 \mapsto \omega^2 :=  \frac{3 \phi_2}{c - \phi_2} -1\in (0,\infty)
\end{equation}
is one-to-one and onto for $\phi_2 \in (c/4,c)$. Therefore, the mapping $a \mapsto L = \mathcal{L}(a,b_-(a)) \in (0,\infty)$ is one-to-one and onto at the boundary $b = b_-(a)$. For every $L \in (0,\infty)$, there exists a unique root of $L = \mathfrak{L}(a,b_-(a))$, which we denote by $a_L$.

Thus, for every fixed $c > 0$ and $L > 0$, there exists exactly one $L$-periodic solution on the left and right boundaries. Since $\mathfrak{L}(a,b)$ is smooth in $(a,b)$ and it is strictly increasing in $b$ by Lemma \ref{theorem-increasing}, the existence of the $C^1$ mapping 
$a \mapsto b = \mathcal{B}_L(a)$ for $a \in (0,a_L)$ follows by the implicit function theorem for $\mathfrak{L}(a,b) = L$ for every fixed $L > 0$. Indeed, $\partial_a \mathfrak{L} + \mathcal{B}_L'(a) \partial_b \mathfrak{L} = 0$ and since $\partial_b \mathfrak{L} > 0$, $\mathcal{B}_L'(a)$ is uniquely defined for every $a \in (0,a_L)$. Since $\phi$ is smooth with respect to parameters by Lemma \ref{lem-trav}, the mapping $a \mapsto \phi = \Phi_L(\cdot,a) \in H^{\infty}_{\rm per}$ is $C^1$ along the curve $b = \mathcal{B}_L(a)$.
\end{proof}

\begin{remark}
	\label{remark-fixed-period}
	The mapping $b \mapsto \phi = \Psi_L(\cdot,b) \in H^{\infty}_{\rm per}$ may not be $C^1$ along the curve $b = \mathcal{B}_L(a)$ because of the non-monotonicity of $\mathfrak{L}(a,b)$ with respect to $a$ shown in  Lemma \ref{lem-nonmonotonicity}. In particular, the mapping $b \mapsto \phi = \Psi_L(\cdot,b) \in H^{\infty}_{\rm per}$ is not $C^1$ at the point where $\mathcal{B}_L'(a) = 0$.
\end{remark}

We next characterize the negative and zero eigenvalues of the Hessian operator $\mathcal{L}$ under the two constraints defining $X_0$ given by (\ref{orth-cond}). The restriction of $\mathcal{L}$ onto $X_0$ is denoted by $\mathcal{L} |_{X_0}$ 
with the corresponding notations $n(\mathcal{L}|_{X_0})$ for the number of negative eigenvalues, taking into account their multiplicities, and $z(\mathcal{L}|_{X_0})$ for the multiplicity of the zero eigenvalue. 
The following lemma gives the count of negative and zero eigenvalues under the two constraints. 

\begin{lemma}
	\label{lem-count}
	Let $a \mapsto b = \mathcal{B}_L(a)$ and $a \mapsto \phi = \Phi_L(\cdot,a) \in H^{\infty}_{\rm per}$ be the $C^1$ mappings of Lemma \ref{lemma-fixed-period}. Assume that $\mathcal{B}_L'(a) \neq 0$ and denote 	
	$$
	\mathcal{M}_L(a) := M(\Phi_L(\cdot,a)) \quad \mbox{\rm and} \quad 
	\mathcal{F}_L(a) := F(\Phi_L(\cdot,a)),
	$$
	where $M(u)$ and $F(u)$ are given by (\ref{Mu}) and (\ref{Fu}). Then, $n(\mathcal{L}|_{X_0}) = 0$ and $z(\mathcal{L}|_{X_0}) = 1$ if and only if the mapping 
	\begin{equation}
	\label{slope-criterion}
	a \mapsto \frac{\mathcal{F}_L(a)}{\mathcal{M}_L(a)^3}
	\end{equation}
	is strictly decreasing and,  for  $\mathcal{B}_L'(a) < 0$, additionally, the mapping $a \mapsto \mathcal{M}_L(a)$ is strictly increasing. 
\end{lemma}

\begin{proof}
Recall that the counting formulas for the negative and zero eigenvalues of $\mathcal{L}|_{X_0}$, see e.g.~\cite{NPL1,NLP2} and references therein, are given by 
\begin{equation}
\label{count-neg}
\left\{ \begin{array}{l}
n(\mathcal{L}|_{X_0}) = n(\mathcal{L}) - n_0 - z_0, \\
z(\mathcal{L}|_{X_0})= z(\mathcal{L}) + z_0,
\end{array} \right.
\end{equation}
where $n_0$ and $z_0$ are the numbers of negative and zero  eigenvalues (counting their multiplicities) of the matrix of projections
\begin{equation}
\label{matrix-P}
S := \left[ \begin{matrix}\langle \mathcal{L}^{-1} 1, 1 \rangle & 
\langle \mathcal{L}^{-1} \phi^2, 1 \rangle \\
\langle \mathcal{L}^{-1} 1, \phi^2 \rangle & 
\langle \mathcal{L}^{-1} \phi^2, \phi^2\rangle \end{matrix} \right].
\end{equation}
It follows from (\ref{der-phi}) with $\phi = \Phi_L(\cdot;a)$ being the smooth $L$-periodic solution along the curve $b = \mathcal{B}_L(a)$ 
with $\mathcal{B}_L'(a) \neq 0$ that 
\begin{equation}
\label{matrix-P-inverse}
\mathcal{L}^{-1} 1 = 4 \partial_b \phi, \quad \mathcal{L}^{-1} \phi^2 = -2c \partial_c \phi - 2b \partial_b \phi.
\end{equation} 
For each part of the curve $b = \mathcal{B}_L(a)$ for which $\mathcal{B}_L'(a) \neq 0$ we introduce the inverse mapping $a = \mathcal{B}_L^{-1}(b)$ 
and redefine $\Phi_L(\cdot;\mathcal{B}_L^{-1}(b)) \equiv \Phi_L(\cdot;b)$, $\mathcal{M}_L(\mathcal{B}_L^{-1}(b)) \equiv \mathcal{M}_L(b)$, 
and  $\mathcal{F}_L(\mathcal{B}_L^{-1}(b)) \equiv \mathcal{F}_L(b)$. 
Due to (\ref{Mu}), (\ref{Fu}), and (\ref{matrix-P-inverse}), matrix $S$ in (\ref{matrix-P}) can be rewritten in the form
\begin{eqnarray*}
S = \left[ \begin{matrix} 4 \partial_b \mathcal{M}_L & -2 c \partial_c \mathcal{M}_L - 2b \partial_b \mathcal{M}_L \\
8	\partial_b \mathcal{F}_L & -4c \partial_c \mathcal{F}_L  
- 4b \partial_b \mathcal{F}_L \end{matrix} \right],
\end{eqnarray*}
so that we obtain
\begin{equation}
\label{det-P}
\det(S) = 16 c \left[ \partial_c \mathcal{M}_L \partial_b \mathcal{F}_L 
- \partial_b \mathcal{M}_L \partial_c \mathcal{F}_L \right].
\end{equation}
Due to the scaling transformation (\ref{scaling}), we can write 
\begin{equation}
\label{scal-transform}
b = c^2 \beta, \quad \Phi_L(\cdot;b) = c \hat{\Phi}(\cdot;\beta), \quad 
\mathcal{M}_L(b) = c \hat{\mathcal{M}}_L(\beta), \quad 
\mathcal{F}_L(b) = c^3 \hat{\mathcal{F}}_L(\beta),
\end{equation}
where $\beta$ and the hat functions are $c$-independent. 
Substituting the transformation (\ref{scal-transform}) into (\ref{det-P}) yields
\begin{align*}
\det(S) & = 16 c^2 \left[ \hat{\mathcal{M}}_L(\beta) \hat{\mathcal{F}}_L'(\beta) - 3 \hat{\mathcal{F}}_L(\beta) \hat{\mathcal{M}}_L'(\beta) \right] \\
&= 16 c^2 \hat{\mathcal{M}}_L(\beta)^4 \frac{d}{d\beta} \left[ \frac{\hat{\mathcal{F}}_L(\beta)}{\hat{\mathcal{M}}_L(\beta)^3} \right].
\end{align*}

Recall that $\partial_a \mathfrak{L} + \mathcal{B}_L'(a) \partial_b \mathfrak{L} = 0$
and $\partial_b \mathfrak{L} > 0$ by Lemma \ref{theorem-increasing}.
For the part of the curve $b = \mathcal{B}_L(a)$ with $\mathcal{B}_L'(a) > 0$, 
we have $\partial_a \mathfrak{L} < 0$ so that $n(\mathcal{L}) = 1$ 
and $z(\mathcal{L}) = 1$ by Lemma \ref{lem-Floquet}. If 
\begin{equation}
\label{stab-criterion-1}
\frac{d}{d\beta} \left[ \frac{\hat{\mathcal{F}}_L(\beta)}{\hat{\mathcal{M}}_L(\beta)^3} \right] < 0,
\end{equation}
then $\det(S) < 0$ so that $S$ has one positive and one negative eigenvalue. 
Then, $n_0 = 1$ and $z_0 = 0$ so that the counting formulas (\ref{count-neg}) give $n(\mathcal{L} |_{X_0}) = 0$ and $z(\mathcal{L}|_{X_0}) = 1$. 
Since $\mathcal{B}_L'(a) > 0$ for this part of the curve $b = \mathcal{B}_L(a)$, 
the criterion (\ref{stab-criterion-1}) is equivalent to the condition that the mapping (\ref{slope-criterion}) is strictly decreasing.

For the part of the curve $b = \mathcal{B}_L(a)$ with $\mathcal{B}_L'(a) < 0$, 
we have $\partial_a \mathfrak{L} > 0$ so that $n(\mathcal{L}) = 2$ 
and $z(\mathcal{L}) = 1$ by Lemma \ref{lem-Floquet}. If 
\begin{equation}
\label{stab-criterion-2}
\frac{d}{d\beta} \left[ \frac{\hat{\mathcal{F}}_L(\beta)}{\hat{\mathcal{M}}_L(\beta)^3} \right] > 0 \quad \mbox{\rm and} \quad \frac{d}{d\beta} \hat{\mathcal{M}}_L(\beta) < 0,
\end{equation}
then $\det(S) > 0$ in view of the first condition, and therefore the symmetric matrix $S$ has two negative eigenvalues since it is negative definite in view of the second condition. 
Then, $n_0 = 2$ and $z_0 = 0$ so that the counting formulas (\ref{count-neg}) give $n(\mathcal{L} |_{X_0}) = 0$ and $z(\mathcal{L}|_{X_0}) = 1$. 
Since $\mathcal{B}_L'(a) < 0$ for this part of the curve $b = \mathcal{B}_L(a)$, 
the criterion (\ref{stab-criterion-2}) is equivalent to the condition that the mapping (\ref{slope-criterion}) is strictly decreasing and the mapping $a \mapsto \mathcal{M}_L(a)$ is strictly increasing by the chain rule. 
\end{proof}

\begin{remark}
	\label{remark-stability}
	It is well-known (see, e.g., \cite{HK08}) that if $n(\mathcal{L}|_{X_0}) = 0$ and $z(\mathcal{L}|_{X_0}) = 1$, then the spectrum of $J \mathcal{L}$ in $L^2_{\rm per}$ is located on the imaginary axis, which implies that the $L$-periodic wave is spectrally stable. Indeed, 
	let $w \in {\rm Dom}(J \mathcal{L}) \subset L^2_{\rm per}$ be the eigenvector of the spectral problem $J \mathcal{L} w = \lambda w$ for the eigenvalue $\lambda \in \mathbb{C}$. By the same computations as in Lemma \ref{lem-constraints}, we have $w \in X_0$ if $\lambda \neq 0$. For every $w \in {\rm Dom}(J \mathcal{L}) \cap X_0$, we obtain 
\begin{equation*}
\lambda \langle \mathcal{L} w,w \rangle = 
\langle \mathcal{L} J \mathcal{L} w,w \rangle = 
-\langle \mathcal{L} w, J \mathcal{L} w \rangle 
= - \bar{\lambda} \langle \mathcal{L} w,w \rangle, 
\end{equation*}
so that 
\begin{equation*}
(\lambda + \bar{\lambda}) \langle \mathcal{L} w,w \rangle = 0.
\end{equation*}
Assume that $\langle \mathcal{L} w,w \rangle = 0$. Since $w \in X_0 \subset L^2_{\rm per}$, the conditions 
$n(\mathcal{L}|_{X_0}) = 0$ and $z(\mathcal{L}|_{X_0}) = 1$ imply 
that $\langle \mathcal{L} w,w \rangle = 0$ can be satisfied 
if and only if $w \in {\rm Ker}(\mathcal{L})$ which contradicts $\lambda \neq 0$. Hence, 
$\langle \mathcal{L} w,w \rangle > 0$, which implies that $\lambda + \bar{\lambda} = 0$ and so 
$\lambda \in i \mathbb{R}$.
\end{remark}

Finally, we confirm the validity of the stability criterion of Lemma \ref{lem-count} for every point in a neighborhood of the boundary 
$a_-(b)$ where the periodic solution is constant. 

\begin{lemma}
	\label{lem-Stokes}
Fix $c > 0$, $b \in (-c^2,\frac{1}{8}c^2)$ and denote $a_L:=a_-(b)$ for fixed period $L = \frac{2\pi}{\omega}$, where $\omega^2$ is given by (\ref{lin-freq}). There exists $\varepsilon > 0$ such that for every $a \in (a_L-\varepsilon,a_L)$, the mapping (\ref{slope-criterion}) is strictly decreasing and the mapping $a \mapsto \mathcal{M}_L(a)$ is strictly increasing. 
\end{lemma}

\begin{proof}
	Let us parameterize the boundary $a_-(b)$ by $\phi_2 \in (c/4,c)$ as 
	in (\ref{a-minus}). We substitute
\begin{equation}
\label{substitution-stokes}
	\phi = \phi_2 (1 + \varphi), \quad a = \phi_2 (c - \phi_2)^3 (1 + \alpha)
\end{equation}
for a function $\varphi$ and a scalar $\alpha$	into (\ref{second-order}) and obtain 
	\begin{equation}
	\label{eq-stokes}
		\varphi'' - \varphi + \frac{1+\alpha}{(1-\eta^{-1} \varphi)^3} - 1 = 0,
	\end{equation}
	where $\eta := (c-\phi_2)/\phi_2 \in (0,3)$ as in the proof of Lemma \ref{theorem-increasing}. The period $L > 0$ is fixed for fixed $c > 0$ and $\phi_2\in (c/4,c)$ by $L = \frac{2\pi}{\omega}$, where $\omega^2$ is given by (\ref{lin-freq}). 	We use the Stokes expansion for even, $L$-periodic solutions with their maximum at $x = 0$, see also  \cite{NPL1,NLP2},
	$$
	\varphi(x) = A \cos(\omega x) + A^2 \varphi_2(x) + A^3 \varphi_3(x) + \mathcal{O}(A^4), \quad \alpha = A^2 \alpha_2 + \mathcal{O}(A^4),
	$$
	where $A>0$ and $\varphi_2, \varphi_3$ are even, $L$-periodic functions. Substituting this expansion into the linearization of \eqref{eq-stokes} and using the definition of $\omega$ in \eqref	{lin-freq},  we obtain  a sequence of compatibility conditions at each order,
\begin{align*}
\mathcal{O}(A^2) :& \quad \varphi_2'' + \omega^2 \varphi_2 + 6 \eta^{-2} \cos^2(\omega x) + \alpha_2 = 0, \\
\mathcal{O}(A^3) :& \quad \varphi_3'' + \omega^2 \varphi_3 + 12 \eta^{-2} \cos(\omega x) \varphi_2 + 10 \eta^{-3} \cos^3(\omega x) + 3 \alpha_2 \eta^{-1} \cos(\omega x) = 0, 
\end{align*}
from which the correction terms can be found. The solution  to the inhomogeneous equation at the order $\mathcal{O}(A^2)$ is given by 
\begin{align*}
\varphi_2(x) = - \frac{3 \eta^{-2} + \alpha_2}{\omega^2} + \frac{\eta^{-2}}{\omega^2} \cos(2\omega x),
\end{align*}
where the  solutions of the homogeneous equation $\varphi_2'' + \omega^2 \varphi_2 = 0$ 
have been set to zero due to the arbitrariness of the parameter $A$. To ensure that the solution to the inhomogeneous equation at the order $\mathcal{O}(A^3)$ is $L$-periodic and not unbounded, we have to remove the term $\cos(\omega x)$ from the source term. After substituting the solution $\varphi_2$ found in the previous step and recalling that 
\begin{align*}
2 \cos(\theta) \cos(2\theta) &= \cos(3\theta) + \cos(\theta), \\ 4\cos^3(\theta) &= \cos(3\theta) + 3\cos(\theta),
\end{align*}
we find that this is the case  if and only if $\alpha_2 = - \frac{5}{2 \eta^2}$. 

Note that if $a \in (a_L-\varepsilon,a_L)$ and the period $L > 0$ is fixed 
along a curve in the $(a,b)$-plane, see Figure \ref{fig-dependence}, then 
the small deviation in $a$ implies a small deviation in $\alpha = A^2 \alpha_2 + \mathcal{O}(A^4)$ in view of \eqref{substitution-stokes}, and hence also in $A$ since $\alpha_2$ is fixed. 

The next step is to expand $\mathcal{M}_L(a)$ and $\mathcal{F}_L(a)$ in terms of $A^2$. Note that we can write these expressions in terms of the new variable $\varphi$ and find that 
\begin{align*}
\mathcal{M}_L(a) &= \phi_2 \left( L + \oint \varphi dx \right), \\
\mathcal{F}_L(a) &= \frac{1}{6} \phi_2^3 \left( L + 3 \oint \varphi dx + 3 \oint \varphi^2 dx + \oint \varphi^3 dx \right). 
\end{align*}
After straightforward computations we obtain that the expansions in terms of small $A^2$ are given by
\begin{align*}
\mathcal{M}_L(a) &= \phi_2 L \left( 1 - \frac{1}{2 \eta^2 \omega^2} A^2 + \mathcal{O}(A^4) \right), \\
\mathcal{F}_L(a) &= \frac{1}{6} \phi_2^3 L \left( 1 + \frac{3}{2} (1 - \eta^{-2} \omega^{-2}) A^2 + \mathcal{O}(A^4) \right),
\end{align*}
so that 
\begin{align*}
\frac{\mathcal{F}_L(a)}{\mathcal{M}_L(a)^3} = \frac{1}{6 L^2} \left[ 1 + \frac{3}{2} A^2 + \mathcal{O}(A^4) \right].
\end{align*}
Since $\alpha_2 < 0$, we have $\frac{d a}{dA^2} < 0$. It follows 
from $\frac{d \mathcal{M}_L(a)}{d A^2} < 0$ so that the mapping $a \mapsto \mathcal{M}_L(a)$ is strictly increasing. 
Similarly, $\frac{d}{d A^2} \frac{\mathcal{F}_L(a)}{\mathcal{M}_L(a)^3} > 0$ 
so that the mapping (\ref{slope-criterion}) is strictly decreasing. 
\end{proof}

\begin{remark}
	The proof of Theorem \ref{theorem-stability} is complete with the results 
	of Lemma \ref{lemma-fixed-period}, Lemma \ref{lem-count}, Remark \ref{remark-stability}, and Lemma \ref{lem-Stokes}.
\end{remark}

\begin{remark}
	It is tempting to conjecture, similarly to what was proven for the CH equation \cite{GMNP}, that the monotonicity of the mapping (\ref{slope-criterion}) along the entire curve with $b = \mathcal{B}_L(a)$ is the only energy stability criterion needed for Theorem \ref{theorem-stability}, whereas the information on the monotonicity of the mapping $a \mapsto \mathcal{M}_L(a)$ is unnecessary and the exceptional point $\mathcal{B}_L'(a) \neq 0$ is irrelevant. However, we are not able to prove this conjecture by only using properties of the Hessian operator $\mathcal{L}$, which is related to the differential equation (\ref{nu-phi-equation}). The successful strategy in \cite{GMNP} relies on the linearized operator for the second-order equation \eqref{second-order}, 
	which is related to the alternative Hamiltonian formulation of the CH equation that is missing for the DP equation, unfortunatley. As a result, the linearized operator for the second-order equation \eqref{second-order} does not define a linearized evolution in Hamiltonian form and therefore, positivity of this operator under two constraints no longer implies spectral stability of  smooth periodic waves. For this reason we have not replicated the strategy of \cite{GMNP} for the CH equation here, but instead rely on the standard Hamiltonian formulation of the DP equation.
\end{remark}


\bibliographystyle{unsrt}

\begin{thebibliography}{99} 
	
\bibitem{Cam} R. Camassa, D. Holm, ``An integrable shallow water equation with peaked solitons", Phys. Rev. Lett. {\bf 71} (1993) 1661--1664.

\bibitem{Chic} C. Chicone, ``The monotonicity of the period function for planar Hamiltonian vector fields", \textit{J. Diff. Equat.} 
{\bf 69} (1987), 310–-321.

\bibitem{Const4} A. Constantin and D. Lannes, ``The hydrodynamical relevance of the Camassa--Holm and Degasperis--Procesi equations", Arch. Ration. Mech. Anal. {\bf 192} (2009) 165--186.

	
\bibitem{CS-02} A. Constantin and W.A. Strauss, ``Stability of the Camassa--Holm solitons”, J. Nonlinear Sci. {\bf 12} (2002), 415--422.



\bibitem{Cox2007} D.A. Cox, J.~Little, and D.~O'Shea
{\em Ideals, varieties, and  algorithms. An introduction to computational algebraic geometry and
  commutative algebra}, Undergraduate Texts in Mathematics, Springer, New York, 2007.

\bibitem{dhh} A.~Degasperis, D. D.~Holm, and A. N. W.~Hone, 
``A new integrable equation with peakon solutions", 
Theor. and Math. Phys. {\bf 133} (2002) 1461--1472.

\bibitem{dhh-proc} A. Degasperis, D. D. Holm, and A. N. W. Hone, ``Integrable and non-integrable equations with peakons", in: {\em Proceedings of Nonlinear Physics — Theory and Experiment II} (World Scientific, Singapore, 2002), pp. 37--43.


\bibitem{dp} A. Degasperis and  M. Procesi, ``Symmetry and Perturbation Theory", in {\em Asymptotic Integrability} (A. Degasperis and G. Gaeta, editors) 
(World Scientific Publishing, Singapore, 1999), pp. 23--37.

\bibitem{Dullin} H. R. Dullin, G. A. Gottwald, and D. D. Holm, ``An integrable shallow water equation with linear and nonlinear
dispersion", Phys. Rev. Lett. {\bf 87} (2001) 194501.


\bibitem{Fulton1984}
W.~Fulton, {\em Introduction to intersection  theory in algebraic geometry}, CBMS Regional Conference
  Series in Mathematics, volume 54, American Mathematical Society, Providence, RI, 1984.


\bibitem{Vil2014}
A.~Garijo and J.~Villadelprat,``Algebraic and analytical tools for the study of the period
  function", J.~Differential Equations {\bf 257} (2014) 2464--2484.

\bibitem{GGG} A. Gasull, H. Giacomini, and J. D. García-Saldana,
``Bifurcation values for a family of planar vector fields of degree five", \textit{Discrete and Continuous Dynamical Systems} {\bf 35} (2015) 669--701.

\bibitem{Gasull1997} A.~Gasull, A.~Guillamon, and V.~Ma\~{n}osa, 
``An Explicit Expression of  the First Liapunov and Period Constants with Applications",
\textit{J.~Math.~Anal.~Appl.} {\bf 211} (1997) 190--212.


	
\bibitem{GMNP}  A. Geyer, R.H. Martins, F. Natali, and D.E. Pelinovsky, ``Stability of smooth periodic traveling waves in the Camassa-Holm equation",
Stud. Appl. Math. {\bf 148} (2022) 27--61 

\bibitem{Geyer-Ostrovsky} A. Geyer and D.E. Pelinovsky, ``Spectral instability of the peaked periodic wave in the reduced Ostrovsky equation", Proceedings of AMS {\bf 148} (2020) 5109--5125.

\bibitem{Geyer-Pel2019} A. Geyer and D.E. Pelinovsky, ``Linear instability and uniqueness of the peaked periodic wave in the reduced Ostrovsky equation", SIAM J. Math. Anal. {\bf 51} (2019) 1188--1208.


\bibitem{Geyer} A. Geyer and D.E. Pelinovsky, ``Spectral stability of periodic waves in the generalized reduced Ostrovsky equation",  \textit{Lett. Math. Phys.} {\bf 107} (2017), 1293--1314.

\bibitem{Geyer2015b} A.~Geyer and J.~Villadelprat, ``On the wave length of smooth periodic traveling waves of the Camassa-Holm equation", J.~Diff.~Eq {\bf 259} (2015), 2317--2332.



\bibitem{GS-book} S. J. Gustafson and I. M. Sigal, {\em Mathematical Concepts of Quantum Mechanics} (Springer, Berlin, 2006).

\bibitem{HK08} M. H$\check{a}$r$\check{a}$gu\c{s} and  T. Kapitula, ``On the spectra of periodic waves for infinite-dimensional Hamiltonian systems",  \textit{Physica D} {\bf 237} (2008), 2649--2671.



\bibitem{rossen}  R. I.~Ivanov, ``Water waves and integrability'',
Phil.~Trans.~R.~Soc.~A  {\bf 365}  (2007) 2267--2280. 

\bibitem{J} M. Johnson, ``Nonlinear stability of periodic traveling wave solutions of the generalized Korteweg-de Vries equation", SIAM J. Math. Anal. {\bf 41} (2009), 1921--1947.


\bibitem{Johnson} R. S. Johnson, ``Camassa--Holm, Korteweg--de Vries and related models for water waves", J. Fluid Mech. {\bf 455} (2002) 63--82.

\bibitem{Kato} T. Kato, ``Perturbation of continuous spectra by trace class operators", Proc. Japan Acad. {\bf 33}
(1957), 260--264.

\bibitem{LP-21} S. Lafortune and D.E. Pelinovsky, ``Spectral instability of peakons in the $b$-family of the Camassa-Holm equations", SIAM J. Math. Anal. {\bf 54} (2022) 4572--4590

\bibitem{LP-22} S. Lafortune and D.E. Pelinovsky, ``Stability of smooth solitary waves in the $b$-Camassa--Holm equations", Physica D {\bf 440} (2022) 133477 (10 pages)


\bibitem{Len1} J. Lenells, ``Traveling wave solutions of the Degasperis--Procesi equation", J. Math. Anal. Appl. {\bf 306} (2005) 72--82.

\bibitem{Len2} J. Lenells, ``Traveling wave solutions of the Camassa--Holm equation", J. Diff. Eqs. {\bf 217} (2005) 393--430

\bibitem{Len3} J. Lenells, ``Stability for the periodic Camassa--Holm equation", Math Scand. {\bf 97} (2005) 188--200.

\bibitem{Liu-21} J. Li, Y. Liu, and Q. Wu, ``Spectral stability of smooth solitary waves for the Degasperis--Procesi equation", 
J. Math. Pures Appl. {\bf 142} (2020) 298--314.

\bibitem{Liu-22} J. Li, Y. Liu, and Q. Wu, ``Orbital stability of smooth solitary waves for the Degasperis--Procesi equation", 
Proc. AMS.  to appear.



\bibitem{Long} T. Long and C. Liu, ``Orbital stability of smooth solitary waves for the $b$-family of Camassa--Holm equations", preprint (2022).

\bibitem{Lund} H. Lundmark and J. Szmigielski, ``A view of the peakon world through the lens of approximation theory", Physica D {\bf 440} (2022) 133446 (44 pages)

\bibitem{Neves} A. Neves, ``Floquet’s Theorem and the stability of periodic waves", J. Dyn. Diff. Equat. {\bf 21} (2009) 555--565.

\bibitem{MP-2021} A. Madiyeva and D. E. Pelinovsky, ``Growth of perturbations to the peaked periodic waves in the Camassa-Holm equation", SIAM J. Math. Anal. {\bf 53} (2021) 3016--3039.

\bibitem{Natali} F. Natali and D. E. Pelinovsky, ``Instability of $H^1$-stable peakons in the Camassa-Holm equation", J. Diff. Eqs. {\bf 268} (2020) 7342--7363.

\bibitem{NPL1} F. Natali, U. Le, and D. E. Pelinovsky, ``New variational characterization of periodic waves in the fractional Korteweg-de Vries equation", Nonlinearity {\bf 33} (2020) 1956--1986

\bibitem{NLP2} F. Natali, U. Le, and D.E. Pelinovsky, ``Periodic waves in the modified fractional Korteweg--de Vries equation", J. Dyn. Diff. Equat. {\bf 34} (2022) 1601--1640.

\bibitem{SB} J. Stoer and R. Bulirsch, 
``Introduction to Numerical Analysis", Springer Verlag, New York Heidelberg, 1980.

\bibitem{Escher2008} J. Escher and Z. Yin, 
``Well-posedness, blow-up phenomena, and global solutions for the b-equation", J. Reine Angew. Math. {\bf 624} (2008) 51--80.

\end{thebibliography}

\end{document}